\documentclass[psamsfont,10pt]{amsart}

\usepackage{xypic}
\usepackage{hyperref}
\usepackage{dsfont}\let\mathbb\mathds
\usepackage[english]{babel} 

\newtheorem{theorem}{Theorem}[section]
\newtheorem{lemma}[theorem]{Lemma}
\newtheorem{corollary}[theorem]{Corollary}
\newtheorem{convention}[theorem]{Convention}
\newtheorem{proposition}[theorem]{Proposition}

\newtheorem{conjecture}{Conjecture}

\newtheorem{problem}{Problem}

\theoremstyle{remark}
\newtheorem{rem}[theorem]{Remark}
\theoremstyle{definition}
\newtheorem{definition}[theorem]{Definition}

\DeclareMathOperator{\im}{Im}

\newcommand{\nts}[1]{\marginpar{#1}}
\renewcommand{\nts}[1]{}

\def\bfG{\mbox{{\bf G}}}       \def\bfH{\mbox{{\bf H}}} \def\bfD{\mbox{{\bf D}}}    \def\bfone{\mbox{{\bf 1}}}

\def\mF{\mathcal{F}}    \def\mP{\mathcal{P}}   \def\mS{\mathcal{S}} \def\mL{\mathcal{L}}   \def\mS{\mathcal{S}}      \def\mC{\mathcal{C}}    \def\mZ{\mathcal{Z}} 

   \def\bC{\mathbb C} \def\bR{\mathbb R} \def\bQ{\mathbb Q}   \def\bZ{\mathbb Z}   \def\bP{\mathbb P}

\def\brp{\mathbb R \mathbb P}

\def\co{\colon\thinspace}

\usepackage{marvosym}

\def\del{\partial}              

\def\kt{\mathfrak{t}}

\def\Reeb{\mbox{Reeb}}
\def\vol{d\varpi}
\def\length{d\sigma}

\def\ra{\rightarrow}

\newcommand{\defin}[1]{\textit{#1}}

\begin{document}

\title[Existence and non uniqueness of cscS metrics]{Existence and non uniqueness of constant scalar curvature toric Sasaki metrics}
\author{Eveline Legendre}
\thanks{This paper is part of the author's Ph.D. thesis. The author would like to thank her supervisors, Vestislav Apostolov and Paul Gauduchon. She is also grateful to Charles Boyer for his comments on a previous version.}

\address{Eveline Legendre : D\'epartement de Math\'ematiques, UQAM \\
C.P. 8888, Succ. Centre-ville Montr\'eal (Qu\'ebec), H3C 3P8, Canada // Centre de Math\'ematiques Laurent Schwartz, \'Ecole Polytechnique, 91128 Palaiseau, France}
\email{eveline.legendre@cirget.ca}


\begin{abstract} We study compatible toric Sasaki metrics with constant scalar curvature on co-oriented compact toric contact manifolds of Reeb type of dimension at least $5$. These metrics come in rays of transversal homothety due to the possible rescaling of the Reeb vector fields. We prove that there exist Reeb vector fields for which the transversal Futaki invariant (restricted to the Lie algebra of the torus) vanishes. Using existence result of~\cite{TGQ}, we show that a co-oriented compact toric contact $5$--manifold whose moment cone has $4$ facets admits a finite number of rays of transversal homothetic compatible toric Sasaki metrics with constant scalar curvature. We point out a family of well-known toric contact structures on $S^2\times S^3$ admitting two non isometric and non transversally homothetic compatible toric Sasaki metrics with constant scalar curvature. \end{abstract}


\maketitle
\section{Introduction}
In this paper we study the existence and uniqueness of compatible Sasaki metrics of constant scalar curvature (cscS for short) on a compact co-oriented contact manifold $(N,\bfD)$, where the uniqueness should be understood up to a contactomorphism and transversal homothety (rescaling of the Reeb vector field). Sasaki--Einstein metrics, which occur when the first Chern class $c_1(\bfD)$ of the contact distribution $\bfD$ vanishes, have been intensively studied in recent years by many authors, see~\cite{BG:book}. On the other hand, the theory of cscS metrics can be viewed as an odd dimensional analogue of the more classical subject of constant scalar curvature K\"ahler metrics, which has been actively studied since the pioneering works of Calabi~\cite{calabi}. We will focus in this paper on the special case when the contact structure is toric of Reeb type in the sense of~\cite{contactNOTE} and the compatible metric is invariant under the torus action. In this setting, the problem of existence of cscS metrics is very closely related to the theory of constant scalar curvature K\"ahler metrics on toric varieties, recently developed by Donaldson in~\cite{don:scalar}.\\

Banyaga and Molino, Boyer and Galicki, and Lerman~\cite{BanyagaMolino1,BanyagaMolino2,contactNOTE,L:contactCONVEX,L:contactToric} classified toric contact manifolds $(N^{2n+1},\bfD, \hat{T}^{n+1})$ (in what follows, we suppose $n>1$).
The action of $\hat{T}$ pull-backs to a Hamiltonian action on the symplectization $(M^{2n+2},\hat{\omega})$ of $(N,\bfD)$, commuting with the Liouville vector field $\tau$, see~\cite{L:contactToric}. In particular, the contact moment map $\hat{\mu}:M\ra (\bR^{n+1})^*$ refers to the unique moment map on the toric symplectic cone which is homogeneous of degree $2$ with respect to the Liouville vector field $\tau$ (i.e $\mL_{\tau}\hat{\mu} = 2\hat{\mu}$) and $\mC=\im \hat{\mu}\cup \{0\}$ is {\it the moment cone}.

 In order to study toric Sasaki metrics, it is not restrictive to consider toric contact manifolds of \textit{Reeb type}, so that there exists a vector $b\in \bR^{n+1}=\mbox{Lie }T^{n+1}$ inducing a Reeb vector field $X_b\in\Gamma(TM)$, see~\cite{BG:book}. Equivalently, $b$ lies in $\mC_+^*$, the interior of the dual cone of $\mC$ (the set of strictly positive linear maps on $\im \hat{\mu} =\mC\backslash \{0\}$). In particular, $\mC$ is a {\it strictly convex} polyhedral cone, that is, $\mC_+^*$ is not empty. From~\cite{contactNOTE,L:contactToric} we know that toric contact manifolds of Reeb type of dimension at least $5$ are in correspondence with strictly convex polyhedral cones $\mC\subset\bR^{n+1}$ which are \textit{good} with respect to a lattice $\Lambda$. This means that every set of primitive vectors normal to a face of $\mC$ can be completed to a basis of $\Lambda$.

Given a strictly convex polyhedral cone $\mC$, which is good with respect to a lattice $\Lambda$, one can associate to any
$b\in\mC_+^*$ the {\it characteristic labeled polytope}\footnote{In~\cite{reebMETRIC}, $\Delta$ is the {\it characteristic polytope} of a toric Sasaki metric with Reeb vector field $X_b$.} $(\Delta_b,u_b)$ where
$$\Delta_b=\mC\cap \left\{ y \,\left|\, \langle b,y\rangle =\frac{1}{2}\right.\right\}$$ is a compact simple polytope and $u_b=\{u_{b1},\dots, u_{bd}\}$ is the set of equivalence classes in $\bR^{n+1}/\bR b$ of the primitive vectors of $\Lambda$ which are inward normal to the facets of $\mC$. Here, $\bR^{n+1}/\bR b$ is identified with the dual vector space of the annihilator of $b$ in $(\bR^{n+1})^*$ which, in turn, is identified with the hyperplane $\{ y\,|\, \langle b,y\rangle =\frac{1}{2}\}$. 

  \begin{rem} Referring to $(\Delta_b,u_b)$ as a labeled polytope is slightly abusive: When there is a lattice $\Lambda'\subset \bR^{n+1}/\bR b$ containing the normals $u_{bi}$'s, there exist uniquely determined positive integers $m_i$'s such that $\frac{1}{m_i}u_{bi}$ are primitive elements of $\Lambda'$. Then $(\Delta_b, m_1,\dots, m_d)$ is a \defin{rational labeled polytope} in the sense of Lerman--Tolman \cite{LT:orbiToric} and it describes a compact toric symplectic orbifold. This case appears when the Reeb vector field $X_b$ is quasi-regular~\cite{BG:book}.
  \end{rem}
Recall that on a toric symplectic orbifold a compatible K\"ahler metric corresponds to a {\it symplectic potential}, $\phi$, that is, a strictly convex smooth function defined on the interior of the moment polytope $\Delta$, which satisfies certain boundary conditions depending on the labeling $u$~\cite{abreuOrbifold,H2FII,don:estimate,guillMET}. We denote the set of these symplectic potentials by $\mS(\Delta,u)$. Similarly, Martelli, Sparks and Yau~\cite{reebMETRIC} parameterized the set of compatible toric Sasaki metrics in terms of homogeneous smooth functions of degree $1$ on $\mathring{\mC}$, the interior of $\mC$, and subject to boundary and convexity conditions. In particular, a K\"ahler cone metric $\hat{g}$ on $(M,\hat{\omega})$ corresponds to a potential $\hat{\phi}$ on $\mathring{\mC}$. According to the Abreu formula~\cite{abreu,abreuSasakAAcoord} the scalar curvature $s_{\hat{g}}$ is then the pull-back by $\hat{\mu}$ of \begin{equation}\label{abreuForm}
  S(\hat{\phi}) =  -\sum_{i,j=0}^n \frac{\del^2 \hat{H}_{ij}}{\del y_i\del y_j}
\end{equation} where $\hat{H}_{ij}$ is the inverse Hessian of $\hat{\phi}$. Hence, a cscS metric corresponds to a potential $\hat{\phi}$ with $b\in\mC_+^*$ such that $S(\hat{\phi})_{|_{\Delta_b}}$ is constant. This correspondence can be equivalently expressed in terms of symplectic potentials on characteristic labeled polytopes. Indeed, any potential $\phi\in \mS(\Delta_b,u_b)$ on a characteristic polytope $(\Delta_b,u_b)$ canonically determines (and is determined by) a {\it Boothy--Wang potential} $\hat{\phi}$ corresponding to a K\"ahler cone metric $\hat{g}$ on $M$, see~\cite{abreuSasakAAcoord} and \S\ref{sectSURVEYmetric} below. The scalar curvature of $\hat{g}$ is the restriction to $N\subset M$ of the pull-back of
$$S(\hat{\phi})= 4S(\phi) - 4n(n+1),$$
where $S(\phi) = -\sum_{i,j=1}^n \frac{\del^2 H_{ij}}{\del y_i\del y_j}$ with $H_{ij}$ the inverse Hessian of $\phi$, see~\cite{abreuSasakAAcoord}. Furthermore, the scalar curvature of the Sasaki metric is $s_g =4S(\phi) -2n$, see~\cite{BG:book} and \S~\ref{cscSuniq}.\\


A primary obstruction to the existence of cscS metrics is given by the {\it Futaki--Sasaki} or {\it transversal Futaki invariant} of the Reeb vector field introduced by Boyer, Galicki and Simanca in~\cite{BGS}. In the toric case, for any Reeb vector field $X_b$, one can restrict this invariant to the Lie algebra of the torus and obtain a vector $\mF_b \in (\bR^{n+1}/\bR b)^*$ such that $\mF_b= 0$ should a compatible cscS toric metric exist. Thus, we can recast the problem of existence and uniqueness of cscS toric metrics:
 \begin{problem}\label{problemSasak} Given a strictly convex good cone $\mC$, does there exist $b\in\mC_0^*$ such that \begin{align}
  &\mF_b=0\\
  & \exists \phi\in \mS(\Delta_b,u_b), \mbox{ such that } S(\phi) \mbox{ is constant? }
  \end{align} If it exists, is such a $b$ unique up to rescaling? \end{problem}

 Rescaling $b\mapsto \lambda^{-1} b$, equivalently the contact form $\eta_b \mapsto \lambda\eta_b$, leads to a ray of toric Sasaki structures $$g' = \lambda g+(\lambda^2 -\lambda)\eta_b\otimes\eta_b.$$ Such a deformation, called {\it transversal homothety}, changes the scalar curvature as $$s_{g'} = \lambda^{-1}(s_g+2n) -2n,$$ see~\cite{BG:book}. In particular, cscS metrics occur in rays. However, once the Reeb vector $b\in\mC_+^*$ is fixed, the uniqueness of cscS metrics follows from uniqueness of solutions of the extremal K\"ahler equation in $\mS(\Delta_b,u_b)$, see~\cite{guan}, and Lemma~\ref{lemUNIQtransv} below.


In view of Problem~\ref{problemSasak}, the Donaldson--Tian--Yau conjecture~\cite{tian:conj,yau,don:scalar} has a straightforward interpretation in the toric Sasaki case using the notion of polystability of labeled polytopes given by Donaldson in~\cite{don:scalar}\footnote{Donaldson uses a measure on the boundary instead of labels; the two notions are equivalent.}.

\begin{conjecture}\label{conjectureSASAK}
  A compact co-oriented toric contact manifold of Reeb type admits a compatible toric cscS metric if and only if there exists $b\in\mC_+^*$ such that $\mF_b=0$ and $(\Delta_b,u_b)$ is polystable.
\end{conjecture}

Donaldson proved his conjecture~\cite{don:scalar,don:estimate,don:extMcond,don:csc} for compact convex labeled polytopes in $\bR^2$. This immediately implies that Conjecture~\ref{conjectureSASAK} holds true for compact $5$--dimensional toric contact manifolds of Reeb type.\\

The question of existence of toric Sasaki--Einstein metrics, which makes sense on co-oriented compact toric contact manifolds with Calabi--Yau cone (that is, $c_1(\bfD)=0$) is now solved. First, Martelli, Sparks and Yau~\cite{reebMETRIC} proved that the volume functional, defined on the space of compatible Sasaki metrics, only depends on the Reeb vector field and, up to a multiplicative constant, is
$$W(b)=\int_{\Delta_b}\vol.$$ Furthermore, they showed that the Hilbert functional is a linear combination of $W$ and $Z$, where $Z$ is defined for any $\phi\in\mS(\Delta_b,u_b)$ as
$$Z(b)=\int_{\Delta_b}S(\phi)\vol$$ and only depends on the Reeb vector field. They also proved that $Z(b)$ coincides with $W(b)$ up to a multiplicative constant, when restricted to a suitable space of normalized Reeb vector fields, see Remark~\ref{remNormSAKAKIEINSTEIN}. The unique critical point of (the restriction of) $W$ is then the only normalized Reeb vector field with vanishing transversal Futaki invariant.\footnote{Martelli, Sparks and Yau extended their results to the non-toric case in~\cite{MSYvolume}.}

Futaki, Ono and Wang~\cite{FutakiOnoWang} showed, on the other hand, that for such a Reeb vector field Problem~\ref{problemSasak} has always a solution (corresponding to a Sasaki--Einstein metric).

\begin{rem}\label{remNormSAKAKIEINSTEIN} Unlike cscS metrics, there are no rays of transversal homothetic Sasaki--Einstein metrics. Indeed, since the scalar curvature $s_g$ of a Sasaki-Einstein metric satisfies $s_g=2n(2n+1)$, see~\cite{BG:book}, being Sasaki--Einstein prevents the rescaling of the Reeb vector field. In particular, there is an obvious normalization of Reeb vector fields in the search of Sasaki--Einstein metrics. However, Sasaki--Einstein metrics are cscS metrics and thus come in rays of such.\end{rem}

In this paper, we extend the Martelli--Sparks--Yau arguments to toric contact manifolds of Reeb type by showing that, after a suitable normalization of the Reeb vector fields, the critical points of the functional
$$F(b) =\frac{Z(b)^{n+1}}{W(b)^n}$$
coincide with normalized Reeb vectors with vanishing transversal Futaki invariant.


%
%

\begin{theorem} \label{theoExist} A co-oriented toric contact manifold of Reeb type admits at least one ray of Reeb vector fields with vanishing (restricted) transversal Futaki invariant.
\end{theorem}
Unlike the Sasaki--Einstein problem, Reeb vector fields with vanishing transversal Futaki invariant do not necessarily lead to cscS metrics, see e.g.~\cite{don:scalar}. However, as we proved in~\cite{TGQ} any labeled quadrilateral $(\Delta,u)$ with vanishing Futaki invariant admits a symplectic potential $\phi\in\mS(\Delta,u)$ for which $S(\phi)$ is constant. Thus, we obtain

\begin{theorem} \label{TheoS2S3} A co-oriented toric contact $5$--dimensional manifold of Reeb type whose moment cone has $4$ facets admits at least $1$ and at most $7$ distinct rays of transversal homothetic compatible Sasaki metrics of constant scalar curvature.
Moreover, for each pair of co-prime numbers, $(p,q)$, such that $p>5q$, there exist $2$ non-isometric, non transversally homothetic Sasaki metrics of constant scalar curvature compatible with the same contact structure on the Wang--Ziller $5$--dimensional manifold $M_{p,q}^{1,1}$.
\end{theorem}
\noindent More precisely, following~\cite[Corollary 1.6]{TGQ} these metrics are explicitly given in terms of two polynomials of degree at most $3$. The toric contact structure on $M_{p,q}^{1,1}$ is the one described in~\cite{BGS2}. The first part of Theorem~\ref{TheoS2S3} partially answers a question of Boyer~\cite{boyerS3bundleS2}.

The paper is organized as follows: Section~\ref{sectionMETRIC} contains basic notions of toric Sasakian geometry, emphasizing the boundary conditions required for the potential to induce a smooth K\"ahler cone metric. We also give the results we need about uniqueness of cscS metrics. In Section~\ref{sectionREEBfamilyCONE} we give a way to check whether or not a labeled polytope is characteristic of a good cone, we study then the properties of the functional $F$ and prove Theorem~\ref{theoExist}. In Section~\ref{sectEXquadril} we specialize our study to the case of cones over quadrilaterals in $\bR^3$ and prove Theorem~\ref{TheoS2S3}.

\section{Sasaki and transversal K\"ahler toric metrics: A quick review}\label{sectionMETRIC}

A {\it labeled polytope}, $(\Delta,u)$, is a simple compact polytope $\Delta$ in a $n$--dimensional vector space $\kt^*$ which has $d$ codimension $1$ faces, called \defin{facets} and denoted by $F_1,\dots, F_d$; together with a set $u=\{u_1,\dots, u_d\}$ of vectors in $\kt$ which are inward (with respect to $\Delta$) and such that $u_i$ is normal to $F_i$. 
Recall that a polytope is \defin{simple} if each vertex is the intersection of $n$ distinct facets. In what follows, we consider $\Delta$ itself as a face. $(\Delta,u)$ is {\it rational} with respect to a lattice $\Lambda$ if $u\subset \Lambda$.
\begin{definition}\label{defnPolytEquiv}
Two labeled polytopes are \defin{equivalent} if one underlying polytope can be mapped on the other via an invertible affine map, $A\co \mathfrak{t}^*\ra \mathfrak{s}^*$, such that the normals correspond via the differential's adjoint $(dA)^*\co \mathfrak{s}\ra \mathfrak{t}$.
\end{definition}

\subsection{Toric contact and symplectic geometry.}\label{sectSURVEYorbifVSpolyt}\label{sectAFFPolyt}\label{assLabePolyt}
A symplectic cone is a triple $(M,\omega,\tau)$ where $(M,\omega)$ is a symplectic manifold and $\tau\in\Gamma(TM)$ is a vector field generating a proper action of $\bR_{>0}$ on $M$ such that $\mL_{\tau}\omega =2\omega$.
\begin{definition}
A \defin{toric symplectic manifold} is a symplectic manifold $(M,\omega)$ together with an effective Hamiltonian action of a torus $T$, $\rho \co T \hookrightarrow \mathrm{Symp}(M,\omega)$, such that the dimension of $T$ is half the dimension of $M$ and there is a proper $T$--equivariant smooth map $\mu \co M \ra \mathfrak{t}^*=(\mbox{Lie}\,T)^*$ satisfying $d\mu(a) = -\iota_{d\rho(a)}\omega$. The map $\mu$ is unique up to an additive constant and is called the \defin{moment map}. We recall that at $p\in M$, $d\rho(a) = X_a(p)=\frac{d}{dt}_{|_{t=0}}(\exp ta)\cdot p$.
\end{definition}
Recall that, to a co-oriented compact connected contact manifold , $(N^{2n+1}, \bfD)$, corresponds a symplectic cone over a compact manifold, $(\bfD_+^o,\hat{\omega},\tau)$. $\bfD_+^o$ is a connected component of the complement of the $0$--section in $\bfD^o$, the annihilator in $T^*N$ of the contact distribution $\bfD$. $\hat{\omega}=d\lambda$ is the restriction of the differential of the canonical Liouville form $\lambda$ of $T^*N$, and $\tau$ is the Liouville vector field $\tau_{(p,\alpha)}= \frac{d}{ds}_{|_{s=0}} e^{2s}\alpha_p$, so that $\mL_{\tau} \hat{\omega}= 2\hat{\omega}$. 
See~\cite{L:contactToric}, for a detailed description.

\begin{definition}
A (compact) \defin{toric contact} manifold $(N^{2n+1}, \bfD, \hat{T}^{n+1})$ is a co-oriented compact connected contact manifold $(N^{2n+1}, \bfD)$ endowed with an effective action of a (maximal) torus $\hat{T}\hookrightarrow \mbox{Diff}(N)$ preserving the contact distribution $\bfD$ and its co-orientation. Equivalently, the symplectic cone $(\bfD_+^o,\hat{\omega},\tau)$ is toric with respect to the action of $\hat{T}$ and the Liouville vector field $\tau$ commutes with $\hat{T}$. We denote by
$$\hat{\mu}\co \bfD_+^o \ra \hat{\mathfrak{t}}^*=(\mbox{Lie}\, \hat{T})^*$$
the \defin{contact moment map}, that is, the unique moment map of $(\bfD_+^o,\hat{\omega}, \hat{T})$ which is homogeneous of degree $2$ with respect to $\tau$, see~\cite{L:contactToric}.
\end{definition}

\begin{definition}\label{defnGOODcone}
A polyhedral cone is \defin{good} with respect to a lattice $\Lambda$, if any facet $F_i$ has a normal vector lying in $\Lambda$ and, for any face $F_I= \cap_{i\in I} F_i$, $\mbox{span}_{\bZ}\{\hat{u}_i\,|\, i\in I\}= \Lambda\cap \mbox{span}_{\bR}\{\hat{u}_i\,|\, i\in I\}$ where $\hat{u}_i$ is the vector normal to $F_i$, primitive in $\Lambda$.
\end{definition}

Lerman~\cite{L:contactCONVEX,L:contactToric} showed that the image, $\im  \hat{\mu}$, of the contact moment map of a compact toric contact manifold $(N, \bfD, \hat{T})$ does not contain $0$ and $\mC= \im \hat{\mu} \cup \{0\}$ is a convex, polyhedral cone which is good with respect to the lattice of circle subgroups, $\Lambda \subset \hat{\mathfrak{t}}$. $\mC$ is called \defin{the moment cone}.

As mentioned in the introduction, a toric contact manifold of dimension at least $5$ is of \defin{Reeb type}, if the moment cone $\mC$ is strictly convex, that is,
$$\mC^*_+ = \{b\in \hat{\mathfrak{t}}\,|\; \forall x\in \mC\backslash \{0\}, \; \langle b,x\rangle>0\}\neq \{\emptyset\}.$$
Indeed, see~\cite{BG:book}, for any $b\in \mC^*_+$, there is a contact form, $\eta_b$, (i.e $\ker \eta_b =\bfD$) for which $X_b$ is a Reeb vector field meaning that: $ \eta_b(X_b)\equiv 1$ and $\mL_{X_b}\eta_b \equiv0.$

Any strictly convex good polyhedral cone is the moment cone of a toric contact manifold of Reeb type, unique up to contactomorphisms, see~\cite{L:contactToric}.
 \begin{definition}\label{defnTransversal}
Let $(\mC,\Lambda)$ be a strictly convex rational polyhedral good cone with $d$ facets. Denote by $\hat{u}_1, \dots, \hat{u}_d$ the set of primitive vectors in $\Lambda$ normal to the facets of $\mC$.  For $b\in\mC^*_+$, we define the labeled polytope
\begin{equation}\label{transversPolyt}
 (\Delta_b, u_b)=( \mC\cap \kt_b^*, [\hat{u}_1]_b,\dots , [\hat{u}_d]_b)
\end{equation} where $\kt_b^*$ the hyperplane $\kt_b^* =\{ x\in \hat{\mathfrak{t}}^*\,|\, \langle b,x\rangle=\frac{1}{2}\}$. Up to translation, $\kt_b$ is the annihilator of $b$ in $\hat{\mathfrak{t}}^*$ and, thus, its dual space is identified with $\hat{\mathfrak{t}}/\bR b$. The polytope $\Delta_b$ is $n$--dimensional and simple polytope. We say that $(\Delta_b, u_b)$ is the {\it characteristic labeled polytope} of $(\mC,\Lambda)$ at $b$.
\end{definition}

As Boyer and Galicki showed in~\cite{contactNOTE}, the space of leaves $\mZ_b$ of the Reeb vector field $X_b$ is an orbifold if and only if the orbits of $X_b$ are closed, that is, if and only if $\bR b\cap \Lambda \neq\{0\}$. In that case, $X_b$ is said \defin{quasi-regular} and $b$ generates a circle subgroup of $\hat{T}$, $T_b = \bR b / (\bR b \cap \Lambda)$, which acts on the cone $(\bfD_+^o,\hat{\omega})$ via the inclusion $\iota_b \co T_b \hookrightarrow \hat{T}$, with moment map $\mu_b =\iota_b^*\circ\hat{\mu}\co \bfD_+^o \ra \bR$. The space of leaves is identified with the symplectic reduction $\hat{\mu}_b^{-1}(1/2)/T_b$. Via the Delzant--Lerman--Tolman~\cite{delzant:corres,LT:orbiToric} correspondence, $(\mZ_b,d\eta_b, T/T_b)$ is the toric symplectic orbifold associated to $(\Delta_b, u_b)$.

\begin{rem}\label{DLTcorresp}  Any compact toric symplectic orbifold admits a moment map whose image is a convex polytope $\Delta$ in the dual of the Lie algebra of $\mathfrak{t}^*$. The weights of the action determine a set of normals $u\subset \kt$ so that $(\Delta,u)$ is rational with respect to the lattice $\Lambda=\ker\exp(\kt\ra T)$. The Delzant--Lerman--Tolman correspondence states that a compact toric symplectic orbifold is determined by its associated rational labeled polytope, up to a $T$--invariant symplectomorphism of orbifolds.

Conversely, any rational labeled polytope can be obtained from a toric symplectic orbifold, via \defin{Delzant's construction}. Two Hamiltonian actions $(\mu,T,\rho)$, $(\mu',T',\rho')$ on a symplectic orbifold $(M,\omega)$ are equivalent if and only if their labeled polytopes are equivalent in the sense of Definition~\ref{defnPolytEquiv}.
\end{rem}

A \defin{toric Sasaki} manifold $(N, \bfD,g,\hat{T})$ is a $(2n+1)$--dimensional Sasaki manifold whose underlying contact structure is toric with respect to an $(n+1)$-dimensional torus $\hat{T}$ and whose metric is $\hat{T}$--invariant. Recall that it corresponds to the K\"ahler toric cone $(M,\hat{\omega},\hat{g},\tau, \hat{T},\hat{\mu})$ where $M=\bfD^o_+$ and $\hat{g}$ is the cone metric of $g$, that is, $\hat{g}$ is homogeneous of degree $2$ with respect to the Liouville vector field $\tau$ and coincides with $g$ on $N\subset M$, seen as the level set $\hat{g}(\tau,\tau)=1$. $(\hat{\omega},\hat{g},\hat{J})$ is a toric K\"ahler structure such that $\tau$ is real holomorphic. Notice that $\hat{J}\tau$ is induced by an element $b\in \mC^*_+\subset\hat{\mathfrak{t}}$, so that $X_b= \hat{J}\tau$ restricts to a Reeb vector field on $N$. The {\it transversal K\"ahler geometry} of $g$ refers to the metric $\check{g}$ induced on $\bfD$ by $g= \eta_b\otimes \eta_b +\check{g}$.

\subsection{K\"ahler metric in action-angle coordinates}\label{sectSURVEYmetric}
The material of this section is taken in~\cite{abreuOrbifold,abreuSasakAAcoord,BGL,CDG,guillMET}. Let $(M^{2n},\omega, J, g, T, \mu)$ be a K\"ahler toric orbifold, that is $g$ is a $T$--invariant K\"ahler metric and $J$ is a complex structure such that $g(J\cdot,\cdot)=\omega(\cdot,\cdot)$. We are interested in the cases where $\mP =\im \mu$ is a strictly convex cone or a polytope. We denote by $\mathring{\mP}$ the interior of $\mP$. Recall from~\cite{delzant:corres,LT:orbiToric,L:contactToric} that $\mathring{M}=\mu^{-1}(\mathring{\mP})$ is the subset of $M$ where the torus acts freely. The K\"ahler metric provides a horizontal distribution for the principal $T$--bundle $\mu \co \mathring{M} \ra \mathring{\mP}$ which is spanned by the vector fields $JX_u$, $u\in \mathfrak{t}=\mbox{Lie } T$. This gives an identification between the tangent space at any point of $\mathring{M}$ and $\mathfrak{t} \oplus \mathfrak{t}^*$. Usually, one chooses a basis $(e_1,\dots,e_n)$ of $\mathfrak{t}$ to identify $\mathring{M} \simeq \mathring{\mP}\times T$ using the flows of the induced commuting vector fields $X_{e_1}$, ... $X_{e_n}$, $JX_{e_1}$, ... $JX_{e_n}$. The action-angle coordinates on $\mathring{M}$ are local coordinates $(\mu_1,\dots,\mu_d,t_1,\dots,t_d)$ on $\mathring{M}$ such that $\mu_i = \langle \mu,e_i\rangle$ and $X_{e_i} =\frac{\del}{\del t_i}$. The differentials $dt_i$ are real-valued closed $1$--forms globally defined on $\mathring{M}$ as dual of $X_{e_i}$ (i.e $dt_i(X_{e_i}) =\delta_{ij}$ and $dt_i(JX_{e_j})=0$).

In action-angle coordinates, 
the symplectic form becomes $\omega =\sum_{i=1}^n d\mu_i\wedge dt_i$. It is well-known \cite{guillMET} that any toric K\"ahler metric is written as
\begin{align}\label{ActionAnglemetric}
g= \sum_{s,r} G_{rs}d\mu_r\otimes d\mu_s + H_{rs}dt_r\otimes dt_s
\end{align}
where the matrix valued functions $(G_{rs})$ and $(H_{rs})$ are smooth on $\mathring{\mP}$, symmetric, positive definite and inverse to each other. 
Following~\cite{H2FII} we define the $S^2\mathfrak{t}^*$--valued function $\bfH \co \mathring{\mP} \ra \mathfrak{t}^*\otimes\mathfrak{t}^*$ by $\bfH_{\mu(p)}(u,v)= g_p(X_u,X_v)$, put $H_{rs}= \bfH(e_r,e_s)$ and $\bfG \co\mathring{\mP} \ra \mathfrak{t}\otimes\mathfrak{t}$ is defined similarly.



When $M$ is compact, necessary and sufficient conditions for a $S^2\mathfrak{t}^*$--valued function $\bfH$ to be induced by a globally defined toric K\"ahler metric on $M$ were established in \cite{abreuOrbifold,H2FII,don:estimate}. We are going to adapt the point of view of \cite{H2FII} to K\"ahler cones. Let $(\Delta,u)$ be a labeled polytope. For a non-empty face $F = F_I= \cap_{i\in I} F_i$ of $\Delta$, denote $\mathfrak{t}_F = \mbox{span}_{\bR}\{ u_i\,|\; i\in I\}$. Its annihilator in $\mathfrak{t}^*$, denoted $\mathfrak{t}_F^o$, is naturally identified with $(\mathfrak{t}/\mathfrak{t}_F)^*$. We use the identification $T^*_{\mu}\kt^* \simeq \kt$.

\begin{definition}\label{defSympPot}The set of symplectic potentials $\mS(\Delta,u)$ is the space of smooth strictly convex functions on $\mathring{\Delta}$ for which $\bfH=(\mbox{Hess }\phi)^{-1}$ is the restriction to $\mathring{\Delta}$ of a smooth $S^2\mathfrak{t}^*$--valued function on $\Delta$, still denoted by $\bfH$, which verifies the boundary condition: For every $y$ in the interior of the facet $F_i\subset \Delta$,
\begin{equation}\label{condCOMPACTIFonH}
  \bfH_y (u_i, \cdot\,) =0\;\;\;\mbox{ and }\;\;\; d\bfH_y (u_i, u_i) =2u_i
\end{equation} and the positivity conditions: The restriction of $\bfH$ to the interior of any face $F \subset \Delta$ is a positive definite $S^2(\mathfrak{t}/\mathfrak{t}_F)^*$--valued function. \end{definition}

\begin{proposition}\label{prop1H2FII}\cite[Proposition 1]{H2FII}
Let $\bfH$ be a positive definite $S^2\mathfrak{t}^*$--valued function on $\mathring{\Delta}$. $\bfH$ comes from a K\"ahler metric on $M$ if and only if there exists a symplectic potential $\phi \in\mS(\Delta,u)$ such that $\bfH=(\mbox{Hess }\phi)^{-1}$.
\end{proposition}

This result follows from two lemmas: \cite[Lemma 2]{H2FII} which holds regardless the completeness of the metric and thus can be used as such here (Lemma~\ref{lemma2H2FII} below) and \cite[Lemma 3]{H2FII}, which we adapt to toric K\"ahler cones (Lemma~\ref{lemAAcoordNOTIMPORTANT} below). Then, we prove Proposition~\ref{propHbSasakREVERSE} which is the adaptation of \cite[Proposition 1]{H2FII}.

\begin{rem}
For this adaptation, it is more natural to work with $S^2\mathfrak{t}$--valued functions $\bfG=(G_{ij})$ without requiring that it is the Hessian of a potential $\phi \in C^{\infty}(\mathring{\mP})$. The metric $g$ defined via (\ref{ActionAnglemetric}) is then a {\it almost K\"ahler} metric.
\end{rem}

\begin{lemma}\cite[Lemma 2]{H2FII} \label{lemma2H2FII}
Let $(M,\omega)$ be a toric symplectic $2n$--manifold or orbifold with moment map $\mu\co M\ra \mP\subset\kt^*$ and suppose that $(g_0, J_0)$, $(g, J)$ are compatible almost K\"ahler metrics on $\mathring{M}=\mu^{-1}(\mathring{\mP})$ of the form (\ref{ActionAnglemetric}), given by $\bfG_0$, $\bfG$ and the same angular coordinates, and such that $(g_0, J_0)$ extends to an almost K\"ahler metric on $M$. Then $(g, J)$ extends to an almost K\"ahler metric on $M$ provided that $\bfG \bfG_0$ and $\bfG_0\bfH \bfG_0 - \bfG_0$ are smooth on $\mP$. \end{lemma}


\begin{lemma} \label{lemAAcoordNOTIMPORTANT}
Let $(M,\hat{\omega},\tau)$ be a toric symplectic cone over a compact co-oriented contact manifold of Reeb type with two cone K\"ahler metrics inducing the same $S^2\hat{\mathfrak{t}}$--valued function $\hat{\bfG}$ on the interior of the moment cone. Then there exists an equivariant symplectomorphism of $(M,\hat{\omega})$ commuting with $\tau$ and sending one metric to the other.
\end{lemma}
 \begin{proof} Let $\hat{g}$ and $\hat{g}'$ be two such metrics, they share the same level set $$\{p\in M\,|\,\hat{g}_p(\tau,\tau)=1\}=\{p\in M\,|\,\hat{g}'_p(\tau,\tau)=1\}\simeq N.$$
Indeed, the part $(\hat{\mu}_1,\dots, \hat{\mu}_n)$ of the action-angle coordinates does not depend on the metric, the Liouville vector field is $\tau=\sum_i 2\hat{\mu}_i\frac{\del}{\del \hat{\mu}_i}$ and thus
$$\hat{g}(\tau,\tau)=4\sum_{i,j=0}^n\hat{\mu}_i\hat{\mu}_j\hat{G}_{ij}= \hat{g}'(\tau,\tau).$$
The equivariant symplectomorphism, say $\psi$, of $(\mathring{M},\hat{\omega})$ sending one set of action-angle coordinates to the other commutes with $\tau$ and sends one metric to the other. Moreover, it restricts to $\mathring{N}$ as an equivariant isometry between $g$ and $g'$ which can then be uniquely extended to a smooth equivariant isometry on $N$ by a standard argument. Finally, since $g$ and $g'$ determine the respective cone metrics, this isometry pull-backs to a global isometry on $M$ which coincides with $\psi$ on $\mathring{M}$.
 \end{proof}

The fact that $(\hat{\omega},\hat{g},\hat{J})$ is cone K\"ahler structure with respect to $\tau$ reads, in terms of $\hat{\bfH}$, as follows:

\begin{lemma}\cite{reebMETRIC} \label{lemReebH}
For any $0\leq i,j\leq n$, $\hat{H}_{ij}$ is homogeneous of degree $1$ with respect to the dilatation in $\kt^*$. Moreover, if $b\in \hat{\mathfrak{t}}$ induces the Reeb vector field $X_b=J\tau$, then for all $\hat{\mu}\in \mC$, $\hat{\bfH}_{\hat{\mu}}(b,\cdot)=2\hat{\mu}$.
\end{lemma}

 \begin{proposition}\label{propHbSasakREVERSE}
Let $(M,\hat{\omega},\hat{T},\tau)$ be a toric symplectic cone associated to a strictly convex good polyhedral cone $(\mC,\Lambda)$ having inward primitive normals $\hat{u}_1$, ... $\hat{u}_d$. Let $\hat{\bfH}$ be a positive definite $S^2\hat{\mathfrak{t}}^*$--valued function on $\mathring{\mC}$. $\hat{\bfH}$ comes from a $\hat{T}$--invariant almost K\"ahler cone metric $\hat{g}$ on $M$ if and only if
\begin{itemize}
  \item $\hat{\bfH}$ is the restriction to $\mathring{\mC}$ of a smooth $S^2\hat{\mathfrak{t}}^*$--valued function on $\mC$,
  \item for every $y$ in the interior of the facet $\hat{F}_i\subset \mC$,
\begin{equation}\label{condCOMPACTIFsasak}
  \hat{\bfH}_y (\hat{u}_i, \cdot) =0\;\;\;\mbox{ and }\;\;\; d\hat{\bfH}_y (\hat{u}_i, \hat{u}_i) =2\hat{u}_i,
\end{equation}
\item the restriction of $\hat{\bfH}$ to the interior of any face $\hat{F} \subset \mC$ is a positive definite $S^2(\hat{\mathfrak{t}}/\hat{\mathfrak{t}}_F)^*$--valued function,
\item $\hat{\bfH}$ is homogeneous of degree $1$ with respect to the dilatation in $\kt^*$.
\end{itemize}  \end{proposition}

\begin{proof}
The necessary part follows from~\cite[Proposition 1]{H2FII} and Lemma~\ref{lemReebH} above. For the converse, take a basis of $\kt$ consisting of the Reeb vector $b$ (already determined by $\hat{\bfH}$ on $\mathring{M}$) and the normals associated to an edge of $\mC$. Following the proof of~\cite[Proposition 1]{H2FII}, we consider the first few terms of the Taylor series of the entries of $\hat{\bfH}$ in this basis. Then, we conclude the proof with Lemmas ~\ref{lemma2H2FII} and~\ref{lemAAcoordNOTIMPORTANT}.
\end{proof}

 For $b\in\mC^*_+$, the {\it Boothy--Wang symplectic potential} $\hat{\phi}$ of $\phi\in\mS(\Delta_b,u_b)$ is the function $$\hat{\phi}(\hat{\mu}) =  2\langle \hat{\mu},b\rangle \cdot \phi\left(\frac{\hat{\mu}}{2\langle \hat{\mu},b\rangle}\right) + \frac{\langle \hat{\mu},b\rangle}{2} \log\langle \hat{\mu},b\rangle$$ defined on the interior of the cone $\mC$, see e.g.~\cite{abreuSasakAAcoord}. The following proposition adapts the relation between potentials via reduction given in~\cite{CDG} to the case of cone metrics and non necessarily quasi-regular Reeb vector fields. We give a proof below since this statement does not appear in this form in the literature and is central in our study.

\begin{proposition}Let $(\mC,\Lambda)$ be a strictly convex good cone with $b\in \mC^*_+$. Denote $(\Delta_b,u_b)$ the characteristic labeled polytope of $(\mC,\Lambda)$ at $b$. $\hat{\bfH}$ is a positive definite $S^2\hat{\mathfrak{t}}^*$--valued function on $\mC$ coming from a $\hat{T}$--invariant almost K\"ahler cone metric $\hat{g}$ on $M$ with Reeb vector field $X_b$ if and only if
\begin{equation}\label{reductionH}
  \hat{\bfH}_{\hat{\mu}}= 2\langle \hat{\mu},b\rangle \bfH^b_{\mu} + 2\frac{\hat{\mu} \otimes\hat{\mu} }{\langle \hat{\mu},b\rangle}
\end{equation} where $\mu=\frac{\hat{\mu}}{2\langle \hat{\mu},b\rangle}\in\Delta_b$ and $\bfH^b$ is a positive definite $S^2 (\hat{\mathfrak{t}}/\bR b)^*$--valued function on $\Delta_b$ satisfying the conditions of Proposition~\ref{prop1H2FII} with respect to $(\Delta_b,u_b)$.

Moreover, in that case, $\hat{\bfH}$ is the Hessian's inverse of a symplectic potential $\hat{\phi}$ if and only if $\hat{\phi}$ is the Boothy--Wang symplectic potential of $\phi\in\mS(\Delta_b,u_b)$ whose Hessian's inverse is $\bfH^b$. Finally, $\check{g}(X_a,X_c)= \bfH^b(a,c)$ on $N$.
\end{proposition}

\begin{proof}
  The necessary part of the first affirmation follows from~\cite{TGQ}. For the converse, take a basis $(e_0, e_1, \dots ,e_n)$ of $\hat{\kt}$ such that $e_0=b$, the Reeb vector. Consider the corresponding coordinates $(y_0,y_1,\dots, y_n)$ on $\hat{\kt}^*$ and write $\hat{\bfH}$ and its inverse $\hat{\bfG}$ as matrices using~(\ref{reductionH})

$$\hat{\bfH}=  \!2\!\left(
\begin{array}{c|c}
\!\! y_0\!\! & y_1  \;\dots\;  y_n\\
\hline
\\
y_1 &\\
\vdots & y_0\bfH^b \!+\!\frac{y_iy_j}{y_0}\\
\\
y_n &\\
\end{array}
\right) \;\; \hat{\bfG}= \frac{1}{2y_0^3}\! \left(
\begin{array}{c|c}
\!\! y_0^2+y_iy_jG_{ij}\!\! & -y_0y_iG_{i1} \;\dots\; -y_0y_iG_{in}\\
\hline
\\
 -y_0y_jG_{1j} &\\
\vdots & \; y_0^2\bfG^b \;\\
\\
-y_0y_jG_{nj} &\\
\end{array}
\right)$$
where $\bfG^b$ stands for the inverse of $\bfH^b$. It is then elementary to check that $\hat{\bfH}$ satisfies the conditions of Proposition~\ref{propHbSasakREVERSE} as soon as $\bfH^b$ satisfies the conditions of Proposition~\ref{prop1H2FII} and that $\hat{\bfG}$ is the Hessian of $\hat{\phi}$ if $\bfG^b$ is the Hessian of $\phi$.
\end{proof}
\begin{rem}If $\bR b\cap \Lambda \neq \{0\}$ then the transversal metric $\check{g}$ induced by $\hat{g}$ on the orbifold $\mZ_b$ is associated to $2\langle \hat{\mu},b\rangle\bfH^b$, see~\cite{TGQ}.\end{rem}

\subsection{Toric cscS metrics and uniqueness}\label{cscSuniq}

A Sasaki structure $(N,\bfD,g)$ implies three Riemannian structures: The Riemannian metric $g$ on the contact manifold $N$, the transversal metric $\check{g}$ on the contact bundle and the Riemannian metric $\hat{g}$ on the symplectization $M=\bfD^o_+$. The respective scalar curvature are related to each other by
$$s_{\hat{g}_{|_{N}}} = s_{\check{g}} - 4n(n+1) = s_{g} - 2n(2n+1),$$ see~\cite{BG:book}. On the other hand, the scalar curvature of a toric K\"ahler cone metric $\hat{g}$ with symplectic potential $\hat{\phi}$ is given by the pull-back of $S(\hat{\phi})$ defined by the Abreu formula~(\ref{abreuForm}) and if $\hat{\phi}$ is the Boothy--Wang potential of $\phi\in\mS(\Delta_b,u_b)$ then $$S(\hat{\phi}) (\hat{\mu})= \frac{1}{\langle \hat{\mu},b\rangle}\left(2S(\phi)\left(\frac{\hat{\mu}}{\langle \hat{\mu},b\rangle}\right) - 2n(n+1) \right),$$
see~\cite{abreu, abreuSasakAAcoord}. We now give a straightforward corollary of the uniqueness of the solutions to Abreu's equation on labeled polytopes, due to Guan \cite{guan} and the formulas above.
\begin{lemma}\label{lemUNIQtransv} Up to an equivariant contactomorphism, for each $b \in \mC^*_+$ there exists at most one toric cscS metric with Reeb vector field $X_b$.
\end{lemma}

\begin{proof} If $\hat{\phi}$ and $\hat{\phi}'$ are potentials on the cone $\mC$ associated to K\"ahler cone metrics having the same Reeb vector field $X_b$ then $\hat{\phi}$ and $\hat{\phi}'$ are the Boothy--Wang potentials of functions $\phi$ and $\phi'\in\mS(\Delta_b,u_b)$. $\hat{\phi}$ (respectively $\hat{\phi}'$) defines a toric cscS metric if and only if $S(\phi)$ (respectively $S(\phi')$) is constant. Now, thanks to Guan's uniqueness result~\cite{guan} (recasted in terms of symplectic potentials on labeled polytopes in~\cite{don:csc}), if $S(\phi)$ and $S(\phi')$ are both constant then $\phi-\phi'$ is affine-linear. In that case, $\hat{\phi}-\hat{\phi}'$ is affine-linear and thus $\hat{\phi}$ and $\hat{\phi}'$ define the same K\"ahler cone metric.
\end{proof}
\begin{proposition}\label{lemUNIQ}
Let $(N,\bfD, \hat{T})$ be a co-oriented compact toric contact manifold of Reeb type with contact moment map $\hat{\mu}\co \bfD^o_+ \ra \hat{\kt}^*$ and moment cone $\mC$. Let $g_a$ and $g_b$ be compatible $\hat{T}$--invariant Sasaki metrics on $N$ with respective vector fields $X_a$ and $X_b$. We suppose that $N$ is not a sphere, that its dimension is at least $5$ and that $(N,\bfD, g_a)$ and $(N,\bfD,g_b)$ are not $3$--Sasaki. If $\varphi \co N \ra N$ is a diffeomorphism such that $\varphi^*g_a=g_b$ then $\varphi$ is a contactomorphism and there exist $\psi\in \mathrm{Isom}(N,g_b)$ and $A\in \mathrm{Gl}(\hat{\kt})$, preserving the lattice $\Lambda=\ker({\rm exp }\co \hat{\kt}\ra \hat{T})$, so that
$$\hat{\mu} \circ (\varphi\circ\psi)^*= A^*\circ\hat{\mu}.$$
In particular, $A^*$ is an automorphism of $\mC$ and $Ab=a$.

Conversely, any linear automorphism of $\mC$ whose adjoint preserves the lattice gives rise to a $T$--equivariant contactomorphism $\psi$ such that if $g$ is a compatible toric Sasaki metric on $(N,\bfD,\hat{T})$ then so is $\psi^*g$.
\end{proposition}
\begin{proof}
Since $3$--Sasaki manifolds, spheres and $3$--manifolds are the only manifolds carrying Riemanian metrics compatible with more than one contact structure, see~\cite{BG:book}, under our assumptions, $\mbox{Isom}(N,g_b)\subset \mbox{Con}(N,\bfD)$. Thus $\hat{T}$ and $\varphi^{-1}\circ\hat{T}\circ\varphi$ are tori in $\mbox{Con}(N,\bfD)$. Since a Hamiltonian action on the symplectization $(\bfD^o_+,\hat{\omega})$ induces isotropic distributions,
$\hat{T}$ and $\varphi^{-1}\circ\hat{T}\circ\varphi$ are maximal tori in $\mbox{Isom}(N,g_b)$ which is a compact Lie group since $N$ is compact. In particular, they are conjugate, that is, there exists $\psi\in \mbox{Isom}(N,g_b)$ such that $$\hat{T}= (\varphi\circ\psi)^{-1}\circ\hat{T}\circ\varphi\circ\psi.$$
The differential $A$ at $1\in \hat{T}$ of the automorphism $\tau\mapsto  (\varphi\circ\psi)^{-1}\circ\tau\circ\varphi\circ\psi$ is linear and preserves the lattice. Moreover, $\hat{\mu}\circ (\varphi\circ\psi)^*$ and $A^*\circ \hat{\mu}$ are moment maps for the same Hamiltonian action of $\hat{T}$ on $(\bfD^o_+,\hat{\omega})$ and they are both homogeneous of degree $2$ with respect to the Liouville vector field. Hence, $\hat{\mu}\circ (\varphi\circ\psi)^*=A^*\circ \hat{\mu}.$

Note that $(\varphi\circ\psi)_*X_b =X_a$ since $(\varphi\circ\psi)^*g_a=g_b$ and $\varphi\circ\psi \in \mbox{Con}(N,\bfD)$. The converse follows from the Delzant--Lerman construction.\end{proof}
%
In the proof of Proposition~\ref{lemUNIQ}, the hypothesis that $(N,\bfD,g_b)$ is neither $3$--Sasaki, a sphere nor a $3$--manifold is used to deduce that $\varphi^{-1}\circ\hat{T}\circ\varphi$ is included in $\mbox{Con}(N,\bfD)$. Thus, we can remove this hypothesis by assuming that $\varphi$ is a $\hat{T}$--equivariant contactomorphism. Combined with Lemma~\ref{lemUNIQtransv}, we get:
\begin{proposition}\label{unicity}
Let $(N,\bfD, \hat{T})$ be a compact toric contact manifold of Reeb type of dimension at least $5$. Two toric cscS metrics $g_a$ and $g_b$, with respective Reeb vector fields $X_a$ and $X_b$, coincide up to a combination of $\hat{T}$--equivariant contactomorphism and transversal homothety, if and only if there exists $\lambda>0$ such that $(\lambda\Delta_b,u_b)$ and $(\Delta_a,u_a)$ are equivalent in the sense of Definition~\ref{defnPolytEquiv}.
\end{proposition}

\subsection{Transversal Futaki invariant and extremal affine function}\label{sectFUTAKI=ZETA}

Let $(\Delta,u)$ be a labeled polytope, $\Delta\subset \kt^*$ and $u\subset \kt$. Choosing a basis $(e_1,\dots, e_n)$ of $\mathfrak{t}$ gives a basis $\mu_0= 1$, $\mu_1=\langle e_1,\cdot\,\rangle$, $\dots$, $\mu_n=\langle e_n,\cdot\,\rangle $ of affine-linear functions.
\begin{definition} Let the vector $\zeta = (\zeta_0,\dots,\zeta_n)\in \bR^{n+1}$ be the unique solution of the linear system
\begin{equation}\label{systCHPext}
\begin{split}
\sum_{j=0}^n W_{ij}(\Delta)\, \zeta_j &= Z_i(\Delta,u),\;\;\; i= 0,\dots, n\\
\mbox{with }\quad W_{ij}(\Delta) = \int_{\Delta} \mu_i\mu_j\vol &\quad \mbox{ and }  \quad Z_i(\Delta,u) =2\!\int_{\del \Delta} \mu_i \length,
 \end{split}
\end{equation}
where the volume form $\vol=d\mu_1\wedge \dots \wedge d\mu_n$ and the measure $\length$ on $\del \Delta$ are related by the equality $u_j\wedge \length=-\vol$ on the facet $F_j$.
We call $\zeta_{(\Delta,u)}=\sum_{i=0}^n \zeta_i \mu_i$ the {\it extremal affine function}.
\end{definition}

\begin{rem}
One can also define $\zeta_{(\Delta,u)}$ as the unique affine function such that $\mL_{\Delta,u} (f) =\int_{\del\Delta}f\length - \frac{1}{2}\int_{\Delta}f\zeta_{(\Delta,u)}\vol=0$ for any smooth function $f$, see~\cite{don:scalar}.
\end{rem}
The extremal affine function $\zeta_{(\Delta,u)}$ is the $L^2(\Delta, \vol)$--projection of ``the scalar curvature'' $S(\phi)$ to the space of affine-linear functions, for any symplectic potential $\phi\in\mS(\Delta,u)$. Indeed, integrating~(\ref{abreuForm}) and using~(\ref{condCOMPACTIFonH}) we get
\begin{align}\label{remZETAscal2}
Z_i(\Delta,u)=\int_{\Delta} S(\phi) \mu_i \vol =2\!\int_{\del \Delta} \mu_i\length.
\end{align}
In view of this, if there exists a symplectic potential $\phi\in\mS(\Delta,u)$ such that $S(\phi)$ is an affine-linear function then $S(\phi)=\zeta_{(\Delta,u)}$. In that case, the corresponding toric Sasaki metric is {\it extremal} in the sense of~\cite{BGS}. Thus we get:
\begin{corollary}
On a compact toric contact manifold with good moment cone $(\mC,\Lambda)$, an extremal compatible toric Sasaki metric with Reeb vector field $X_b$ has constant scalar curvature if and only $\zeta_{(\Delta_b,u_b)}$ is constant where $(\Delta_b,u_b)$ is the characteristic labeled polytope of $(\mC,\Lambda)$ at $b$.
\end{corollary}

For any $\phi\in \mS(\Delta,u)$, we define the linear functional $\mF_{(\Delta,u)}\co\kt\ra \bR$ as
$$\mF_{(\Delta,u)}(a)=\int_{\Delta} f_a S(\phi)\vol $$
where $f_a(\mu) = \langle a,\mu\rangle \int_{\Delta} \vol -\int_{\Delta} \langle a,\mu\rangle\vol$ has mean value $0$. Via~(\ref{remZETAscal2}), $\mF_{(\Delta,u)}$ does not depend on the choice of $\phi$. Indeed, setting $\bar{Z}=\frac{1}{\int_{\Delta} \vol}Z_0(\Delta,u)$, $$\mF_{(\Delta,u)}(a)=\int_{\Delta} f_a \zeta_{(\Delta,u)}\vol = \int_{\Delta} f_a (\zeta_{(\Delta,u)} - \bar{Z})\vol.$$ In particular, $\mF_{(\Delta,u)}= 0$ if and only if $\zeta_{(\Delta,u)}$ is constant. \\

Given any compatible $\hat{T}$--invariant CR-structure on a toric contact manifold with Reeb vector field $X_b$, $\mF_b=\mF_{(\Delta_b,u_b)}$ is the restriction of the transversal Futaki invariant of~\cite{BGS} to the space of real transversally holomorphic vector fields which are induced by the toric action. However, our definition of $\mF_{(\Delta_b,u_b)}$ is independent of the choice of a compatible CR-structure, relating this invariant to the symplectic version of the Futaki invariant introduced in~\cite{lejmi}.

\section{The Reeb family of a labeled polytope}\label{sectionREEBfamilyCONE}

A {\it labeled cone} $(\mC,L)$ consists of a polyhedral cone, $\mC$, with $d$ facets in some vector space $V$ and $L=\{L_1,\dots, L_d\}\subset V^*$ so that $\mC=\{ y\,|\, \langle y, L_i\rangle \geq 0\;\; i=1,\dots, d \}$. Moreover, a labeled polytope $(\Delta,u)$ is {\it characteristic} of $(\mC,L)$ at $b\in V^*$, if $\Delta =\mC\cap \{ y\,|\, \langle y, b \rangle=1\}$ and $u_i = [L_i]\in V^*/\bR b$. The interesting case is when $L_1,\dots, L_d$ span a lattice $\Lambda\subset V^*$ for which $\mC$ is good, since in that case $(\Delta,u)$ determines the transversal geometry associated to the Reeb vector $X_b$ on the toric contact manifold associated to $(\mC,\Lambda)$.
\begin{convention}
From now on, we set $\Delta =\mC\cap \{ y\,|\, \langle y, b \rangle=1\}$ instead of $\Delta =\mC\cap \{ y\,|\, \langle y, b \rangle=1/2\}$. This convention facilitates the calculations of the next sections and geometrically corresponds to normalizing the Liouville vector field so the symplectic form of the symplectic cone is homogeneous of degree $1$ instead of $2$.
\end{convention}
\begin{definition} A {\it Reeb family} is a set of equivalence classes of labeled polytopes characteristic of a given labeled cone $(\mC,L)$.\end{definition}

\subsection{The cone associated to a polytope}

Let $(\Delta, u)$ be a labeled polytope with {\it defining functions}:
$$L_i = \langle \,\cdot\, ,u_i\rangle +\lambda_i \in \mbox{Aff}(\kt^*,\bR).$$ That is $\Delta= \{x \in \kt^*\,|\, L_i(x)\geq0,\, i=1,\dots, d\},$ $F_i= \{x\in \kt^*\,|\,L_i(x)=0\}$ and $u_i =dL_i$ via the identification $T^*_\mu\kt^* \simeq \kt$.\\

The defining functions $L_1,\dots, L_d$ determine a cone
$$\mC(\Delta) = \{y\in \mbox{Aff}(\kt^*,\bR)^*\; |\; \langle y, L_i\rangle \geq 0\; i=1,\dots, d\}$$
and its dual $\mC^*(\Delta)=\{ L \in \mbox{Aff}(\kt^*,\bR) \;|\; \langle y, L\rangle \geq 0 \;\; \forall y\in \mC(\Delta) \,\}$. By translating the polytope $\Delta + \mu$ we translate the defining functions $L_i^{\mu}=L_i-\langle \mu,u_i\rangle$, producing a linear equivalence between the cones $\mC(\Delta)$ and $\mC(\Delta+\mu)$. More generally, we have:
\begin{lemma}\label{translationCone} If $(\Delta,u)$ and $(\Delta',u')$ are equivalent by an invertible affine map $A$ then so are  $(\mC(\Delta'),L')$ on $(\mC(\Delta),L)$ via the adjoint map of the pull-back of $A$.
\end{lemma}

\begin{proposition} \label{propTRANSlabeledPOLYTOPE}
A labeled polytope $(\Delta,u)$ is characteristic of a good cone if and only if the defining functions $L_1$,..., $L_d$ of $(\Delta,u)$ span a lattice in ${\rm Aff}(\kt^*,\bR)$ with respect to which the cone $\mC(\Delta)$ is good.\end{proposition}

This proposition follows from the next lemma where $\bfone\in \mbox{Aff}(\kt^*,\bR)$ denotes the constant function equals to $1$.

\begin{lemma}\label{NaturalityCone}
Let $(\Delta,u)$ be a labeled polytope with defining functions $L_1$, ... $L_d$.  $\mC(\Delta)$ is a non-empty, strictly convex, polyhedral cone. Moreover, $(\Delta,u)$ is characteristic of $(\mC(\Delta), \{L_1,\dots, L_d\} )$ at $\bfone \in \mbox{Aff}(\kt^*,\bR)$, and up to linear equivalence, this is the unique labeled cone for which $(\Delta,u)$ is characteristic.
\end{lemma}
\begin{proof}
Consider the evaluation map $e \co \kt^* \hookrightarrow \mbox{Aff}(\kt^*,\bR)^*$, that is for $\mu\in\kt^*$ the map $e_\mu \co L \ra L(\mu)$ is linear. We have $\im e=\{y\,|\, \langle y, \bfone \rangle=1\}$ and then
$e(\Delta) = \im e \cap \mC(\Delta,u).$ For any $y\in \mbox{Aff}(\kt^*,\bR)^*$ there exist a unique $\mu\in \kt^*$ and a unique $r=\langle y, \bfone \rangle\in \bR$ such that $y= e_\mu -e_0 +re_0$ where $e$ is the evaluation map as above. This gives an identification $\mbox{Aff}(\kt^*,\bR)^*\simeq \kt^*\times \bR$ leading to
\begin{equation}
  \begin{split}\label{decompCanonicalCone}
  \mC(\Delta)\backslash\{0\}=\{e_{r\mu} -e_0 +re_0\;|\; \mu \in \Delta,\; r>0\}\simeq \Delta\times \bR_{> 0}.
  \end{split}
\end{equation}
By using again the identification $T_{\mu}^*\kt^* \simeq \kt$, the differential maps $\mbox{Aff}(\kt^*,\bR)$ to $\kt$. The differential corresponds also to the quotient map of by the linear subspace of constant function. Thus $dL_i =[L_i] =u_i \in \mbox{Aff}(\kt^*,\bR)/\bR\bfone.$ The uniqueness part of Lemma~\ref{NaturalityCone} is straightforward.
\end{proof}

\begin{rem}
For a given labeled (simple) polytope $(\Delta,u)$ with vertices $\nu_1$, ... $\nu_N$ and facets $F_1$, ... $F_d$, one can define $N$ lattices of 
$\mbox{Aff}(\kt^*,\bR)$ as $\Lambda_i = \langle L_l \,|\, \nu_i \in F_l \rangle_{\bZ}$. They are free groups of rank $n$ and one can prove that, assuming that $\Lambda= \langle L_l,\dots, L_d \rangle_{\bZ}$ is a lattice in $\rm{Aff}(\kt^*,\bR)$, the cone $(\mC(\Delta),\Lambda)$ is good if and only if all the groups $\Lambda/ \Lambda_1, \dots, \Lambda/ \Lambda_N$ are free.\end{rem}

\subsection{The Reeb family of a labeled polytope} Let $(\Delta,u)$ be a labeled polytope with defining functions $L_1$, ..., $L_d$. Denote the interior of $\mC^*(\Delta)$ by $\mC^*_+(\Delta)$. It is the set of affine maps $b\in\rm{Aff}(\mathfrak{t}^*,\bR)$ which are strictly positive on $\mC(\Delta)$. $\mC^*_+(\Delta)$ is non-empty and open ($\mC(\Delta)$ is a closed subset of $\rm{Aff}(\mathfrak{t}^*,\bR)^*$) and, by using~(\ref{decompCanonicalCone}), is
$$\mC^*_+(\Delta)=\{b\in {\rm Aff}(\mathfrak{t}^*,\bR)\,|\;  \langle b, \mu\rangle >0\; \forall \mu\in \Delta\}.$$

 Using Lemma~\ref{NaturalityCone} we get that: {\it If $(\Delta,u)$ and $(\Delta', u')$ are in the same Reeb family then there exists a unique vector $b\in \mC^*_+(\Delta)$ such that $(\Delta', u')$ is characteristic of $(\mC(\Delta), \{L_1,\dots, L_d\})$ at $b$.} In particular, the cone $\mC^*_+(\Delta)$ provides an effective parametrization of the $\Reeb$ family of $(\Delta,u)$. The affine hyperplane $$\kt^*_b = \{y \in {\rm Aff}(\kt^*,\bR)^* \;|\; \langle y, b\rangle = 1\}$$ is identified with the annihilator of $b$ in $\mbox{Aff}(\kt^*,\bR)^*$ via a translation and then its dual vector space is the quotient $\kt_{b} = \mbox{Aff}(\kt^*,\bR)/\bR b$.

\begin{proposition} \label{explicitReebFamily}
The $\Reeb$ family of $(\Delta,u)$ is parameterized by
$$\left\{(\Delta_b,u_b)=\left.\left(\Psi_b(\Delta), [L_1]_b, \dots, [L_d]_b\right)\; \right| \; b\in \mC^*_+(\Delta) \right\}$$ where $[L_l]_b$ is the equivalence class of $L_l$ in $\kt_b$ and the map $\Psi_b\co\Delta \ra \kt^*_b$ is defined as $\Psi_b(\mu)=\frac{e_{\mu}}{b(\mu)}$.
\end{proposition}

\begin{proof} Using the decomposition~(\ref{decompCanonicalCone}), we see that \begin{equation}\label{Delta_b}
\kt^*_{b}\cap \mC(\Delta) =\left\{ \left.e_{\frac{\mu}{b(\mu)}}-e_0 +\frac{e_0}{b(\mu)} = \frac{e_{\mu}}{b(\mu)} \,\right|\, \mu\in \Delta\right\}.\end{equation}
The map $\Psi_b(\mu)=\frac{e_{\mu}}{b(\mu)}$ is well-defined and injective on any set where $b$ is positive.\\

 A basis, $(v_1, \dots, v_n)$, of $\kt$ provides coordinates on $\kt^*$ via $\mu_i=\langle \mu, v_i\rangle$ and so we write $\mu=(\mu_1,\dots, \mu_n)$, as well as, a basis of $\mbox{Aff}(\kt^*,\bR)$, that is
\begin{equation}\label{coordSYS}
  (\bfone, \langle \cdot, v_1\rangle, \dots, \langle \cdot, v_1\rangle)
\end{equation}
which, in turn, gives coordinates on $\mbox{Aff}(\kt^*,\bR)^*$ as $y=(y_0, y_1,\dots,y_n)$ where $y_0 = \langle y, \bfone\rangle$ and $y_i = \langle y, v_i \rangle$.\\

 In this system of coordinates, we have $e_{\mu} = (1, \mu_1,\dots, \mu_n)$ and thus
$$\Psi_b(\mu) = \frac{1}{b(\mu)}(1, \mu_1,\dots, \mu_n).$$
Denoting $b=b(0)\bfone + \sum_{i=1}^n b_iv_i,$ the differential of $\Psi_b$ at $\mu$ is
\begin{equation}\label{d_bPsi}
d_{\mu}\Psi_b =\frac{1}{b(\mu)^2}\left( -\frac{\del}{\del y_0}\otimes db +\sum_{i,j =1}^n \left( b(\mu)  \frac{\del}{\del y_i}\otimes d\mu_i - b_j\mu_i  \frac{\del}{\del y_i}\otimes d\mu_j\right)  \right)
 \end{equation} where $db$ is identified with an element of $\kt$ and thus with a linear map in $\mbox{Aff}(\kt^*,\bR)$. For a given vector $X=\sum_{i=1}^n X_i\frac{\del}{\del \mu_i} \in T_{\mu} \Delta$ we compute that $d_{\mu}\Psi_b(X) = 0$ if and only if $db(X)=0$ and $b(\mu)X=0$ that is $\Psi_b$ is an immersion.
\end{proof}

\subsection{The Futaki invariant of a $\Reeb$ family}

%
 The purpose of this paragraph is to find a functional over the Reeb family of $(\Delta,u)$, whose critical points are the Reeb vector $b\in\mC^*_+(\Delta)$ such that the extremal affine function $\zeta_{(\Delta_b,u_b)}$ is constant (i.e. the corresponding restricted Futaki invariant $\mF_b$ vanishes, see \S~\ref{sectFUTAKI=ZETA}).\\

\noindent {\bf Notation:} Fix a basis of $\kt$ giving coordinates $(y_0,\dots, y_n)$ on $\mbox{Aff}(\kt_b^*,\bR)^*$ as above and set $\vol=dy_0\wedge dy_1\wedge \dots \wedge dy_n$. For $b\in \mC^*_+(\Delta)$, put $\zeta(b)=\zeta_{(\Delta_b,u_b)}$ where $(\Delta_b,u_b)$ is the labeled polytope in the $\Reeb$ family of $(\Delta,u)$ given by $b$ via the parametrization of Proposition~\ref{explicitReebFamily}. Denote by $\vol_b$ the volume form on $\kt_b^*=\{ y\,|\, \langle y,b\rangle =1\}\subset \mbox{Aff}(\kt^*,\bR)^*$, satisfying $b\wedge\vol_b=\vol$ and $\length_b$ is the measure on $\del \Delta_b$ determined by the equality $L_l \wedge \length_b=-\vol_b$ on the facet $F_{b,l}=\Delta_b\cap L_l^o$. In the system of coordinates induced from $(y_0,\dots, y_n)$ on $\kt_b^*$, the affine extremal function is written $\zeta(b) = \zeta_0(b)+\sum_{i=1}^n\zeta_i(b)y_i\,\in\,\mbox{Aff}(\kt_b^*,\bR)$ where $(\zeta_0(b), \zeta_1(b),\dots,\zeta_n(b))\in \bR^{n+1}$ is the solution of a linear system~(\ref{systCHPext}) involving the functions:
 $$W_{ij}(b) = W_{ij}(\Delta_b)\;\;\mbox{ and }\;\;Z_i(b)= Z_i(\Delta_b,u_b)$$
computed using $\vol_b$.

\begin{rem} In the case where $L_1,\dots,L_d$ span a lattice $\Lambda$ for which $\mC(\Delta)$ is good, there is a contact manifold $(N,D)$ associated to $(\mC(\Delta),\Lambda)$. Then, as in~\cite{reebMETRIC}, one can compute that up to a positive multiplicative constant depending only on the dimension of $N$, $W_{00}(b)$ is the volume of $N$ with respect to the volume form $\eta_b\wedge(d\eta_b)^n$ where $\eta_b$ is the contact form of $X_b$.\end{rem}

\begin{lemma}\label{calculEAFb} For $i,j =1,\dots, n$
\begin{align*}
W_{00}(b) = \int_{\Delta} \frac{1}{b(\mu)^{n+1}}\vol,\;\;\;& W_{ij}(b) =  \int_{\Delta} \frac{\mu_i\mu_j}{b(\mu)^{n+3}}\vol,\\
W_{i0}(b) = W_{0i}(b)= &\int_{\Delta} \frac{\mu_i}{b(\mu)^{n+2}}\vol
\end{align*}
and
\begin{align*}
Z_0(b) =2\int_{\del \Delta} \frac{1}{b(\mu)^{n}} \length,\;\; Z_i(b) =2\int_{\del \Delta} \frac{\mu_i}{b(\mu)^{n+1}} \length.
\end{align*}
\end{lemma}

\begin{proof} We use the coordinates systems on $\kt^*$ and $\mbox{Aff}(\kt^*,\bR)$ introduced in the proof of Proposition~\ref{explicitReebFamily}.

Choose $p=(p_1,\dots,p_n)\in\Delta$ and denote
 \begin{equation*}
\vol_b = \frac{1}{b(p)}\left(dy_1\wedge \dots \wedge dy_n + \sum_{i=1}^n(-1)^{i+1}p_i dy_0\wedge \dots  \wedge\widehat{dy_i}\wedge \dots\wedge dy_n \right).
\end{equation*}
Thus $b\wedge\vol_b=dy_0\wedge dy_1\wedge \dots \wedge dy_n$, since $b=b(0)dy_0 +\sum_{i=1}^n b_idy_i$ when viewed as an element of $T^*_y\mbox{Aff}(\kt^*,\bR)^*$. By setting $\alpha =\sum_{k=1}^n \alpha_k$ with $\alpha_k = (-1)^{k+1}p_k dy_1\wedge \dots  \wedge\widehat{dy_k}\wedge \dots\wedge dy_n$, we get
$$\vol_b = \frac{1}{b(p)}\left(dy_1\wedge \dots \wedge dy_n + dy_0\wedge \alpha\right).$$
On the other hand, put $A_{\mu}=\sum_{i,j =1}^n \left( b(\mu)  \frac{\del}{\del y_i}\otimes d\mu_i - b_i\mu_j  \frac{\del}{\del y_i}\otimes d\mu_j\right)$. In view of the expression of $d_{\mu}\Psi_b$~(\ref{d_bPsi}), $A_{\mu}=b(\mu)^2d_{\mu}\Psi_b + \frac{\del}{\del y_0}\otimes db$. Moreover, $A_{\mu}$ is a morphism between $T_{\mu}\Delta$ and the kernel of $dy_0$ in the tangent space of $\mbox{Aff}(\kt_b^*,\bR)^*$ and one can prove that\footnote{For example, a proof can use induction on $n$ when viewing $A_{\mu}$ as a $n\times n$-- matrix depending on two vectors $\mu\in\bR^n$ and $b\in\bR^{n+1}$.}
\begin{equation*}
  \det A_{\mu} = b(0)b(\mu)^{n-1}\;\;\mbox{ and }\;\;\;db\wedge A_{\mu}^*\alpha_k=-b_kp_kb(\mu)^{n-1}d\mu_1 \wedge \dots\wedge d\mu_n.
\end{equation*}
Hence, since $\alpha$ is a $(n-2)$--form on $\mbox{Aff}(\kt^*,\bR)$ such that $\alpha(\frac{\del }{\del y_0})=0$ we have
\begin{equation*}(\Psi_b^*\vol_b)_{\mu} =\frac{1}{b(p)b(\mu)^{2n}}\left( A^*_{\mu}dy_1\wedge \dots \wedge dy_n  - db\wedge  A^*_{\mu}\alpha\right)= \frac{1}{b(\mu)^{n+1}}\vol.
\end{equation*}
 where $db= \sum_{i=1}^nb_id\mu_i$, via the identification $T^*_\mu \kt^*\simeq \kt$. The first part of the Lemma~\ref{calculEAFb} follows then easily and it remains to prove the statement concerning the functions $Z_0$ and $Z_i$. Note that $\Psi^*_b L_l =\frac{1}{b(\mu)}u_l$ and then
\begin{equation*}
  u_l\wedge(\Psi_b^*\length_b)_{\mu} = b(\mu)\Psi_b^*\left(L_l\wedge\length_b\right)_{\mu} = -b(\nu)(\Psi_b^*\vol)_{\mu} = -\frac{1}{b(\mu)^n}\vol.
\end{equation*}
This shows that $(\Psi_b^*\length_b)_{\mu}= \frac{1}{b(\mu)^n}\length$ for $\mu\in\del\Delta$ which concludes the proof.
\end{proof}

\noindent {\bf Convention:}  For now on, we suppose that $b(0)=1$. There is no loss of generality since, in view of the defining equations~(\ref{systCHPext}), for $r>0$ we have $$\zeta(rb) = r\zeta_0(b)+r^2\sum_{i=1}^n\zeta_i(b)\mu_i.$$ Note that $\Omega =\{b\in\mC^*_+(\Delta)\,|\,b(0)=1\}$ is relatively compact in $\mbox{Aff}(\kt^*,\bR)$.
\begin{proposition} The critical points of the functional $F: \Omega\ra \bR$, defined as $$F(b)= \frac{Z_0(b)^{n+1}}{W_{00}(b)^{n}},$$ are the affine-linear functions $b$ for which $\zeta(b)$ is constant.
\end{proposition}

\begin{proof}
  Notice that $\zeta(b)$ is constant if and only if $\zeta(b)=\zeta_0(b)$ which happens if and only if $\zeta_0(b)$ is solution of the linear system
\begin{equation}\label{linearSYSTzeta(b)=cst}  W_{i0}(b)\, \zeta_0(b) = Z_i(b),\;\;\; i= 0,\dots, n.\end{equation}
In that case, $\zeta_0(b) = \frac{Z_0(b)}{W_{00}(b)}$ with $b$ a solution of:
\begin{equation}\label{EquationZETA(b)=cst} W_{i0}(b)Z_0(b) - W_{00}(b)Z_i(b)=0,\;\;\; i= 1,\dots, n.\end{equation}
 By noticing that for $i= 1,\dots, n$, $$\frac{\del}{\del b_i} W_{00}(b)=-(n+1)W_{i0}(b) \;\;\;\mbox{ and } \;\;\; \frac{\del}{\del b_i} Z_{0}(b)=-nZ_{i}(b),$$ we compute the differential of $F$ at $b\in \Omega$: $$d_bF = \frac{-n(n+1)Z_0(b)^n}{W_{00}(b)^{n+1}}\sum_{i=1}^n(W_{i0}(b)Z_0(b) - W_{00}(b)Z_i(b)) db_i.$$ which concludes the proof.
\end{proof}

\begin{lemma}\label{divW} Let $b\in\Omega$, then $$W_{00}(b)= n\sum_l \int_{F_l}\frac{L_l(0)}{b^n}\length.$$
\end{lemma}
\begin{proof} For $b\in\Omega$, put $X=\sum_i\frac{\mu_i}{b^{n}}\frac{\del}{\del\mu_i}$ so we have $\mbox{div} X= \frac{n}{b^{n+1}}$.
  Thus,

  $$\frac{1}{n}W_{00}(b) = \int_{\del\Delta} \iota_X \vol= -\sum_l\int_{F_l} \frac{\langle u_l,\mu \rangle}{b^n}\length = \sum_l \int_{F_l}\frac{L_l(0)}{b^n}\length$$
since $u_l\wedge \sum_i\iota_{\frac{\del}{\del\mu_i}}\vol = \langle u_l,\mu \rangle \vol$.
\end{proof}

Consider the map $x=(x_1,\dots,x_d)\co\Omega\ra \bR^d$ whose components are the strictly convex and strictly positive functions $$x_l=x_l(b)= \int_{F_l}\frac{1}{b^n}\length.$$
With this notation and by putting $\lambda_l=L_l(0)$, the functionals $W_{00}$, $Z_0$ and $F$ are
\begin{equation}\label{WZFx}
  W_{00} (b) = n\sum_{i=1}^d \lambda_lx_l,\;\;\;Z_{0} (b) = 2\sum_{i=1}^d x_l \;\;\mbox{ and }\;\;F(b)=\frac{\left(2\sum x_l\right)^{n+1}}{(\sum n\lambda_lx_l)^{n}}.
\end{equation}
For each $l=1,\dots, d$, $F_l$ is a $(n-1)$--dimensional polytope whose volume is $x_l$, up to a positive multiplicative constant determined by $u_l$. This suggests to apply Lemma~\ref{divW} recursively. Note that $\int_{E}\frac{1}{b^2}\length(E)=\frac{1}{b(p_E)b(q_E)}$, where $\length(E)$ is the Lebesgue measure on the edge $E$ and $p_E$ and $q_E$ denote the vertices of $\Delta$ lying in $E$. Therefore, for each edge there are suitable constants $\alpha(E)$, $\beta(E)$ so that
  \begin{equation}\label{Wvertices}
W_{00}(b)=\sum_{E \in\,\mathrm{edges}(\Delta)} \frac{\alpha(E)}{b(p_E)b(q_E)} \;\;\mbox{ and }\;\;Z_0(b)= \sum_{E \in\,\mathrm{edges}(\Delta)}\frac{\beta(E) }{b(p_E)b(q_E)}.
\end{equation}
Moreover, $F$ is a rational function of the values of $b$ on the vertices. More precisely,
 \begin{equation}
   F(b)=  \left(\prod_{\nu\in \,\mathrm{vert}(\Delta)}\frac{1}{b(\nu)}\right)\cdot\frac{\left(\sum \beta(E)\prod_{\nu\notin E}b(\nu) \right)^{n+1}}{\left(\sum \alpha(E)\prod_{\nu\notin E}b(\nu) \right)^{n}}
 \end{equation}
 where the sums are taken over edges of $\Delta$.
\begin{proof}[Proof of Theorem~\ref{theoExist}]
Formulas~(\ref{WZFx}) imply that $\lambda=\mbox{max}\{\lambda_l\,|\,l=1,\dots, d\} \, > 0$ and
$$F(b) \geq  \frac{\left(2\sum x_l\right)^{n+1}}{(n\lambda\sum x_l)^{n}}= \frac{2^{n+1}}{(\lambda n)^n} \sum x_l=\frac{2^{n}}{(\lambda n)^n} Z_0(b).$$
Moreover, $\del\Omega$ is the set of affine-linear functions vanishing on the boundary of $\Delta$ (but not on the interior). In particular they vanish on some vertices of $\Delta$. Since $Z_0(b)$ only depends on the value of $b$ on the vertices, see~(\ref{Wvertices}), $Z_0(b)$ and thus $F(b)$ converge to infinity when $b$ converges to a point of $\del \Omega$.

Hence, Theorem~\ref{theoExist} follows from the fact that $F$ is a strictly positive function, defined on a relatively compact open set $\Omega$, and converging to infinity at the boundary. In particular, $F$ must have a critical point somewhere in $\Omega$.
\end{proof}

\begin{proposition}
Let $(\Delta,u)$ be a $n$--dimensional labeled polytope with $N$ vertices. The set of critical points of $F$ is a real algebraic set given as the common roots of $n$ polynomials in $n$ variables of degree $2N-3$. \end{proposition}
\begin{proof} By using~(\ref{Wvertices}), $d_bF =0$ if and only if for any linear function $h\co \kt^*\ra\bR$
\begin{equation}\label{equatdF=0}
\sum_{E} \frac{(nZ_0\alpha(E) - (n+1)W_{00}\beta(E))(h(p_E)b(q_E) +h(q_E)b(p_E))}{b(p_E)^2b(q_E)^2}= 0.
\end{equation}

A linear function is determined by its values on $n$ linearly independent points. Taking a set of $n$ linearly independent vertices of $\Delta$, say $p_1, \dots, p_n$, (\ref{equatdF=0}) may be written as a homogeneous equation in $n$ variables $h(p_1) , \dots, h(p_n)$. In particular, if~(\ref{equatdF=0}) holds for any linear function $h$, the coefficient $P_i$ of $h(p_i)$ in~(\ref{equatdF=0}) vanishes for all $i$. These coefficients $P_i$ are functions of $b(p_1),\dots, b(p_n)$ (since $b(0)=1$, the affine-linear function $b$ is determined by its value at $p_1, \dots,p_n$). 
It is easy to see that, up to a suitable positive multiplicative constant, $P_i$ is a polynomial of degree at most $2N-3$ in $n$ variables $b(p_1),\dots, b(p_n)$.\end{proof}

\subsubsection{The Sasaki--Einstein case}

Let $(N, \bfD)$ be a contact manifold such that $c_1(\bfD)$, the first Chern class of the contact bundle $\bfD$, vanishes. In~\cite{reebMETRIC}, it is shown that the normalized Reeb vector field for which the transversal Futaki invariant is zero corresponds to the critical point of the volume functional and that such a point is unique. In our setting, this implies that, if $c_1(\bfD)=0$, the critical point of $F$ is unique and corresponds to the critical point of $W_{00}(b)$ in $\Omega$.

The condition $c_1(\bfD)=0$ is a necessary condition for the existence of a Sasaki--Einstein metric and corresponds to the fact that the primitive normals of the moment cone lie in a hyperplane, see~\cite{FutakiOnoWang,reebMETRIC}. Moreover, if $X_b$ is the Reeb vector field of a Sasaki--Einstein metric then the basic first Chern class $c_1^B = 2(n+1)[d\eta_b]_B$ which implies that $(\Delta_b,u_b)$ is {\it monotone} in the sense that:

\begin{definition} A labeled polytope $(\Delta,u)$ is \defin{monotone} if there exists $\mu \in \mathring{\Delta}$ such that $L_l(\mu)= c$ for all $l=1,\dots,d$ where $c$ is some positive constant.\end{definition}

In~\cite{do:Large}, integral polytopes which are monotone (with respect to the normals primitive in dual lattice in which lie the vertices of the polytope) are called Fano polytopes. This terminology is justified since a smooth toric variety $X$ is Fano in the usual sense (i.e the anticanonical line bundle $-K_X$ is ample) if and only if the integral Delzant polytope associated to $(X, -K_X)$ is monotone in the sense above. Equivalently, the symplectic manifold associated to this integral Delzant polytope is monotone (i.e the symplectic class coincide up to a multiplicative constant to the first Chern class). In the orbifold case, one can prove that a rational labeled polytope $(\Delta,u)$ is monotone if only if the associated symplectic toric orbifold $(M,\omega)$ is monotone. The next lemma is straightforward.

\begin{lemma}\label{LemmaFanoReebFamily}
The defining functions of a {\it monotone} polytope lie in a hyperplane and any labeled polytope in the Reeb family of a monotone labeled polytope is monotone.
\end{lemma}

\begin{theorem}\cite{reebMETRIC} If $(\Delta,u)$ is monotone then $F$ has a unique critical point. Moreover, up to a translation of $\Delta$, there exists a constant $\lambda>0$ such that $W=\lambda Z$.
\end{theorem}
\begin{proof}
The functional $F$ depends equivariantly on the representative of the affine equivalence class of a labeled polytope. In particular the number of critical points of $F$ does not change if we translate the polytope. If $(\Delta,u)$ is monotone, there exists $\mu\in\mathring{\Delta}$ such that $L_l(\mu)=c>0$ then, by using Lemma~\ref{divW}, the function $W_{00}$ associated to $(\Delta-\mu,u)$ is $W_{00}(b)=  \frac{n}{c} Z_0(b)$.\end{proof}

\section{Examples: The case of quadrilaterals}\label{sectEXquadril}\label{subsectionReebFamilSQUARE}
Up to affine transform there exists a unique strictly convex cone with $4$ facets in $\bR^3$. Indeed, $\bP{\rm Gl}(3,\bR)$ acts transitively on the set of generic $4$--tuples of points of $\brp^2$. In particular, up to linear transform, every Reeb family of quadrilaterals contains the equivalence class of a labeled square. It is then enough to consider the Reeb family of labeled square in order to study Reeb families of quadrilaterals.\\

Let $\Delta_o$ be the convex hull of $p_1=(-1,-1)$, $p_2=(-1,1)$, $p_3=(1,1)$, $p_4=(1,-1)$ in $\bR^2$. $\Delta_o$ is a square and the vectors normal to its edges are of the form
\begin{align}
  \label{eq:lesu}
u_1= \frac{1}{r_1}e_1,\;\;u_2=\frac{-1}{r_2}e_2,\;\;u_3= \frac{-1}{r_3}e_1,\;\;u_4= \frac{1}{r_4}e_2
\end{align}
where $e_1=\begin{pmatrix}
  1\\
  0
\end{pmatrix}$, $e_2=\begin{pmatrix}
  0\\
  1
\end{pmatrix}\in(\bR^2)^*$. The defining functions of $(\Delta_o,u)$ are
$$L_l= \langle\,\cdot\,,u_l\rangle+\frac{1}{r_l}$$
for $l=1,\dots, 4$. Let denote the edges $E_l= L_l^{-1}(0)\cap\Delta_o$. With coordinates $(x,y)$ on $\bR^2$, we have $u_1=\frac{1}{r_1}dx$, $u_2=\frac{-1}{r_2}dy$, $u_3= \frac{-1}{r_3}dx$, $u_4= \frac{1}{r_4}dy$. The following lemma is straightforward.

\begin{lemma} \label{squareFANO}$(\Delta_o,u)$ is monotone if and only if $r_1+r_3=r_2+r_4$.\end{lemma}

Moreover, one can easily check that $$\Omega=\{b=1+b_1e_1+b_2e_2\;|\; |b_1+b_2|<1, |b_1-b_2|<1\}.$$ By integration, we get
$$W_{00}(b)=\int_{\Delta_o}\frac{1}{b^3}dxdy= \frac{2}{b(p_1)b(p_2)b(p_3)b(p_4)}.$$
In this setting, the measure $\length$ can be made explicit:
$$\length_{|_{E_1}}=-r_1dy,\;\; \length_{|_{E_2}}=-r_2dx, \;\;\length_{|_{E_3}}=r_3dy,\;\; \length_{|_{E_4}}=r_4dx.$$ Integrating again leads to
$$Z_{0}(b)=\int_{\del\Delta_o}\frac{1}{b^2}\length= \frac{2r_1}{b(p_1)b(p_2)}+ \frac{2r_2}{b(p_2)b(p_3)}+\frac{2r_3}{b(p_3)b(p_4)}+\frac{2r_4}{b(p_1)b(p_4)}.$$ We then get
$$F(b)=\frac{Z_{0}(b)^3}{W_{00}(b)^2} =\frac{2\left(r_1b(p_3)b(p_4) + r_2b(p_1)b(p_4)+r_3b(p_1)b(p_2)+r_4b(p_2)b(p_3)\right)^3}{b(p_1)b(p_2)b(p_3)b(p_4)} .$$

Put $a_1=b_1-b_2$ and $a_2=b_1+b_2$ so that $b(p_1) =1-a_2$, $b(p_2) =1-a_1$, $b(p_3) =1+a_2$, $b(p_4) =1+a_1$ and $\Omega\simeq\{ (a_1,a_2)\,|\, |a_i| <1 \}$. Moreover,
$$F(b)=\frac{2\left(a_1a_2K+ a_1(K-2(r_2-r_3))+a_2(K-2(r_4-r_3)) + K-2(r_2+r_4)\right)^3}{(1-a_1^2)(1-a_2^2)}$$
with $K=r_1+r_3-r_2-r_4$. Note that $K=0$ if $(\Delta_o,u)$ is monotone.\\
%
%
%
One can prove the following lemma using elementary methods:
\begin{lemma}\label{NbCRITICALbounded}
The point $b=1+b_1e_1+b_2e_2 \in\Omega$ is a critical point of $F$ if and only if $(a_1,a_2)=(b_1-b_2,b_1+b_2)\in\bR^2$ is a common root of polynomials
\begin{equation*}\begin{split}P(a_1,a_2) = -a_1^2a_2 K + a_1^2&(K+2(r_2-r_3)) + 2a_1a_2(K+2(r_4-r_3))\\
&+ 2a_1(K+2(r_2+r_4)) +3a_2K +K+2(r_2-r_3)),
\end{split}\end{equation*}
\begin{equation*}\begin{split}Q(a_1,a_2)=-a_2^2a_1 K + a_2^2&(K+2(r_4-r_3)) + 2a_1a_2(K+2(r_2-r_3))\\
&+ 2a_2(K+2(r_2+r_4)) +3a_1K +K+2(r_4-r_3)).
\end{split}\end{equation*}
lying in the square $\{(a_1,a_2)\,|\, |a_i|<1\}$. Consequently, $F$ has at most $7$ critical points.
\end{lemma}

\begin{lemma}\label{lemCOND2rootsSQUARE}
If $r_1=r_3$, $r_2=r_4$ and $r_1+r_3>5(r_2+r_4)$ then $K>0$ and $P$ and $Q$ have exactly $5$ distinct common roots: $(0,0)$, $\pm(a,-a)$ with $0<a^2 =1-\frac{4(r_2+r_4)}{K} <1$ and $\pm(a,a)$ with $a^2 =5+\frac{4(r_2+r_4)}{K} >1$.

In particular, $F$ admits $3$ distinct critical points in $\Omega$.
\end{lemma}
\begin{proof}
If $r_1=r_3$, $r_2=r_4$ then $K+2(r_4-r_3)=0$, $K+2(r_2-r_3)=0$ and
\begin{align*}
P(a_1,a_2) &= -a_1^2a_2 K  + 2a_1(K+2(r_2+r_4)) +3a_2K, \\
Q(a_1,a_2) &= -a_2^2a_1 K  + 2a_2(K+2(r_2+r_4)) +3a_1K.
\end{align*}
Write $P(x,a_2) =C(x) a_2 +D(x)$ where $C(x) =-K x^2 +3K$ and $D(x)=2(K+2(r_2+r_4))x$. Since $C$ and $D$ do not have any common root, $P$ and $Q$ have at most $5$ common roots of the form
$$\left(a, -\frac{D(a)}{C(a)}\right)\quad \mbox{with $a$ a root of} \quad C^2(x)Q\!\left(x, -\frac{D(x)}{C(x)}\right).$$
By noticing that
\begin{align*}
P(a_1,-a_1) =-Q(a_1,-a_1) &= a_1(a_1^2K - K +4(r_2+r_4))  \quad \mbox{and}\\
P(a_1,a_1) =Q(a_1,a_1) &= a_1(-a_1^2K + 5K +4(r_2+r_4))
\end{align*}
we get the desired $5$ distinct common roots.

Since $r_1+r_3>5(r_2+r_4)$, $0<1-\frac{4(r_2+r_4)}{K}< 1 $ while $5+\frac{4(r_2+r_4)}{K}>1.$
\end{proof}

We have to find out which examples provided by Lemma~\ref{lemCOND2rootsSQUARE} correspond to Reeb vector fields on a toric contact manifold, that is, which labeled squares of Lemma~\ref{lemCOND2rootsSQUARE} are characteristic of a good cone. Denote by $(e_1,e_2,e_3)$ the standard basis of $\bR^3$. Consider the polyhedral cone $\mC_o=\{x\;|\, \langle x,\delta_j\rangle \geq 0,\; \; j=1,\dots, 4\}\subset \bR^3$ where the rays $\delta_i$ are
 $$\delta_1 = \bR_{>0} (e_3 + e_1),\; \delta_2 = \bR_{>0} (e_3 + e_2),\;\delta_3 = \bR_{>0} (e_3-e_1),\; \delta_4 = \bR_{>0} (e_3-e_2).$$

\begin{lemma} \label{lemBONCONE4facets}
Each strictly convex polyhedral cone with $4$ facets in $\bR^3$ is equivalent to $\mC_o$ by an affine transform. Moreover, $\mC_o$ is a good cone with respect to a lattice $\Lambda\subset \bR^3$ if and only if there exists $s>0$ such that the primitive normal vectors inward to $\mC_o$ are $\hat{u}_1=sc_1(e_1+e_3)$, $\hat{u}_2=sc_2(e_2+e_3)$, $\hat{u}_3=sc_3(e_3-e_1)$ and $\hat{u}_4=sc_4(e_3-e_2)$ for some positive integers $c_1$, $c_2$, $c_3$, $c_4$ such that
\begin{equation}\label{conditionPPCM}
  {\rm lcm}(c_1,c_2)= {\rm lcm}(c_1,c_4)={\rm lcm}(c_3,c_2)={\rm lcm}(c_3,c_4).
\end{equation}
\end{lemma}

\begin{proof} The first statement follows from the fact that $\bP\mathrm{Gl}(3,\bR)$ acts transitively on the set of generic $4$--tuples of points of $\brp^{2}$ and that the normal lines of a strictly convex polyhedral cone are generic (any subset of $3$ elements is a basis).

In our case, the normal inward vectors are $\hat{u}_1=c_1(e_1+e_3)$, $\hat{u}_2=c_2(e_2+e_3)$, $\hat{u}_3=c_3(e_3-e_1)$ and $\hat{u}_4=c_4(e_3-e_2)$. Up to multiplication by a constant $s>0$, we may assume that $c_i\in \bQ$, for all $i$. Indeed, the fact that $(\mC_o,\Lambda)$ is rational (i.e it is possible to choose $u_j\in \Lambda$) implies that $c_i/c_j\in\bQ$. 
Moreover, there exists a primitive vector of $\bZ^4$, say $k= (k_1,k_2,k_3,k_4)$, such that $\sum_j k_j u_j =0$. This vector is unique up to sign and, putting $C={\rm  lcm}(c_1, c_2, c_3, c_4)$, we choose it to be $$k= (C/c_1,-C/c_2, C/c_3,-C/c_4).$$ The Delzant--Lerman construction~\cite{L:contactToric} implies that the symplectic cone over the toric contact manifold is the quotient of
$$\tilde{X}= \left\{(z_1,z_2,z_3,z_4)\in\bC^4\backslash \{0\} \;\left|\; \sum_{j=1}^4 k_j |z_j|^2 =0\right.\right\}$$
by the action of $S^1$, defined by $\rho \co S^1 \times \tilde{X} \ra \tilde{X}$, $$\rho(\lambda,z)= ((\lambda^{k_1} z_1,\lambda^{k_3}z_3), (\lambda^{k_2}z_2,\lambda^{k_4}z_4)).$$
Lerman~\cite{L:contactToric} showed that the quotient constructed via this method from a rational polyhedral cone $(\mC_o,\Lambda)$ is smooth if and only if $(\mC_o,\Lambda)$ is a good cone.
The stabilizer of a point $z\in \tilde{X}$ is determined by its vanishing components, precisely $\mbox{Stab}_{\rho} z = \{ \lambda \in S^1 \,|\, \lambda^{k_j}=1 \mbox{ if } z_j\neq 0\}.$ One can then verify that the stabilizer group of each point $z\in \tilde{X}$ is trivial if and only if \begin{equation}\label{conditionPGCD}
  {\rm gcd}(k_1, k_2)= {\rm gcd}(k_1, k_4)= {\rm gcd}(k_3, k_2)= {\rm gcd}(k_1, k_4)=1.
\end{equation} Since $k_j=(-1)^{j-1}C/c_j$ and $|ab|={\rm lcm}(a,b){\rm gcd}(a,b)$, the condition~(\ref{conditionPGCD}) is equivalent to ${\rm lcm}(c_1,c_2)= {\rm lcm}(c_1,c_4)={\rm lcm}(c_3,c_2)={\rm lcm}(c_3,c_4)=C.$\end{proof}

\begin{proof}[Proof of Theorem~\ref{TheoS2S3}]
Let $(r_1,r_2,r_3,r_4)\in\bZ^4$ with $p=r_1=r_3$, $q=r_2=r_4$ being coprime integers such that $p>5q$. Let $u$ be defined by~\eqref{eq:lesu}. The labeled cone of $(\Delta_o,u)$ is identified with the cone $\mC_o$ labeled with the vectors
$$\hat{u}_1=\frac{1}{p}(e_1+e_3),\;\;\hat{u}_2=\frac{1}{q}(e_2+e_3),\;\;\hat{u}_3=\frac{1}{p}(e_3-e_1),\;\;\hat{u}_4=\frac{1}{q}(e_2-e_3).$$
These vectors are all contained in a lattice $\Lambda$ and the rational cone $(\mC_o,\Lambda)$ is good since it satisfies Lemma~\ref{lemBONCONE4facets}. In particular, $(\mC_o,\Lambda)$ is associated to the Wang--Ziller manifold $M^{1,1}_{p,q}$ with the toric contact structure $(\bfD,\hat{T})$ described in~\cite{BGS2}.

The set of characteristic labeled polytopes of $(\mC_o,\Lambda)$ is the Reeb family of $(\Delta_o,u)$ which satisfies the hypothesis of Lemma~\ref{lemCOND2rootsSQUARE}. Hence, there exist 3 Reeb vectors:
$$b_o=(0,0,1)\;\;\mbox{ and }\;\;b_{\pm}= \left(0,\pm\sqrt{1+\frac{4q}{p-q}},1 \right)$$
whose respective characteristic labeled quadrilaterals have constant extremal affine functions. Via~\cite[Theorem 1.4]{TGQ}, there exist symplectic potentials, say $\phi_{b_o}\in\mS(\Delta_{b_o},u_{b_o})$ and $\phi_{\pm b}\in\mS(\Delta_{ b_{\pm}},u_{b_{\pm}})$, for which $S(\phi_{b_o})$, $S(\phi_{b_+})$ and $S(\phi_{b_{-}})$ are constant. In particular, their respective Boothy--Wang symplectic potentials $\hat{\phi}_{b_o}$, $\hat{\phi}_{b_+}$ and $\hat{\phi}_{b_{-}}$ define toric K\"ahler cone metrics $\hat{g}_{b_o}$, $\hat{g}_{b_+}$ and $\hat{g}_{b_{-}}$ on the symplectic cone over $(M^{1,1}_{p,q}, \bfD,\hat{T})$ for which the associated Sasaki metrics $g_{b_o}$, $g_{b_+}$ and $g_{b_{-}}$ on $M^{1,1}_{p,q}$ have constant scalar curvature.

The Wang--Ziller manifold $M^{1,1}_{p,q}$ is diffeomorphic to $S^2\times S^3$, see~\cite{BGS2}, and cannot carry a $3$--Sasaki structure due to its dimension. Thus, it satisfies the hypothesis of Proposition~\ref{lemUNIQ}. So, if there exists a diffeomorphism $\psi$ such that $\psi^*g_{b_o}$ is a transversal homothety of $g_{b_+}$, then there exists a real number $\lambda>0$ such that $(\lambda\Delta_{b_o},u_{b_o})$ is equivalent to $(\Delta_{b_+},u_{b_+})$ in the sense of Definition~\ref{defnPolytEquiv}. But, this cannot happen since $\Delta_{b_o}$ is a square while $\Delta_{b_+}$ is a trapezoid. Similarly $g_{b_o}$ and $g_{b_{-}}$  are not isometric as Riemannian metrics even up to a transversal homothety.

On the other hand, the linear transform $A= -e_1\otimes e_1^* - e_2\otimes e_2^* +e_3\otimes e_3^*$ preserves the set of normals and exchanges $b_+$ and $b_-$. Thus, $A^*$ preserves $\mC_o$ and provides a $\hat{T}$--equivariant contactomorphism sending $g_{b_+}$ to $g_{b_{-}}$, thanks to Proposition \ref{unicity}. \end{proof}

\bibliographystyle{abbrv}

\begin{thebibliography}{DDDD}


 \bibitem{abreu}
  {\scshape M. Abreu}
   \emph{K\"ahler geometry of toric varieties and extremal metrics}, Internat. J. Math. {\bf 9} (1998), 641--651.

 \bibitem{abreuOrbifold}
  {\scshape M. Abreu}
    \emph{K\"ahler metrics on toric orbifolds}, J. Differential Geom. {\bf 58} (2001), 151--187.

\bibitem{abreuSasakAAcoord}
 {\scshape M. Abreu}
   \emph{K\"ahler--Sasaki geometry of toric symplectic cones in action-angle coordinates}, Portugal. Math. (to appear), arXiv:mathDG/0912.0492.

 \bibitem{H2FII}
  {\scshape V. Apostolov, D. M. J. Calderbank, P. Gauduchon, C. T{\o}nnesen-Friedman}
    \emph{Hamiltonian $2$--forms in K\"ahler geometry. II. Global classification}, J. Differential Geom. {\bf 68} (2004), 277--345.



 \bibitem{BanyagaMolino1}
  {\scshape A. Banyaga, P. Molino}
\emph{G\'eom\'etrie des formes de contact compl\`etement integrable de type torique}, S\'eminaire Gaston Darboux, Montpellier, 1991--1992, 1--25.

 \bibitem{BanyagaMolino2}
 {\scshape A. Banyaga, P. Molino}
\emph{Complete Integrability in Contact Geometry}, Pennsylvania State University preprint PM 197, 1996.

 \bibitem{boyerS3bundleS2}
  {\scshape C. P. Boyer}
    \emph{Extremal Sasakian metrics on $S^3$--bundle over $S^2$}, arXiv:mathDG/1002.1049.

 \bibitem{contactNOTE}
  {\scshape C. P. Boyer,  K. Galicki}
    \emph{A note on toric contact geometry}, J. Geom. Phys. {\bf 35} (2000), 288--298.

 \bibitem{BG:book}
  {\scshape C. P. Boyer,  K. Galicki}
    \emph{Sasakian geometry}, Oxford Mathematical Monographs. Oxford University Press, Oxford, 2008.

 \bibitem{BGS}
  {\scshape C. P. Boyer,  K. Galicki, S. R. Simanca}
    \emph{Canonical Sasakian metrics}, Commun. Math. Phys. {\bf 279} (2008), 705--733.

 \bibitem{BGS2}
  {\scshape C. P. Boyer,  K. Galicki, S. R. Simanca}
    \emph{The Sasaki cone and extremal Sasakian metrics}, Proceedings of the conference on Riemannian topology, K. Galicki, S. R. Simanca Eds., 263--290, Birkh\"auser, Boston, 2008.

 \bibitem{BGL}
  {\scshape D. Burns, V. Guillemin, E. Lerman}
    \emph{Kaehler metrics on singular toric varieties}, Pacific J. Math.  {\bf 238} (2008), 27--40.

  \bibitem{calabi}
   {\scshape E. Calabi}
   \emph{Extremal K\"ahler metrics. II.}, Differential geometry and complex analysis, I. Chavel, H. M. Farkas Eds., 95--114, Springer, Berlin, 1985.

 \bibitem{CDG}
  {\scshape D. M. J. Calderbank, L. David, P. Gauduchon}
    \emph{The Guillemin formula and Kahler metrics on toric symplectic manifolds}, J. Symplectic Geom. {\bf 1} (2003), 767--784.


 \bibitem{delzant:corres}
  {\scshape T. Delzant}
    \emph{Hamiltoniens p\'eriodiques et images convexes de l'application moment}, Bull. Soc. Math. France {\bf 116} (1988), 315--339.

 \bibitem{do:Large}
  {\scshape S. K. Donaldson}
    \emph{ K\"ahler geometry on toric manifolds, and some other manifolds with large symmetry}, Handbook of geometric analysis. No. 1, Adv. Lect. Math. (ALM) {\bf 7}, 29--75, Int. Press, Somerville, MA, 2008.

 \bibitem{don:scalar}
  {\scshape S. K. Donaldson}
    \emph{Scalar curvature and stability of toric varieties}, J. Differential Geom. {\bf 62} (2002), 289--349.

 \bibitem{don:estimate}
  {\scshape S. K. Donaldson}
    \emph{Interior estimates for solutions of Abreu's equation}, Collect. Math. {\bf 56} (2005), 103--142.

 \bibitem{don:extMcond}
  {\scshape S. K. Donaldson}
    \emph{Extremal metrics on toric surfaces: a continuity method}, J. Differential Geom. {\bf 79} (2008), 389--432.

 \bibitem{don:csc}
  {\scshape S. K. Donaldson}
      \emph{Constant scalar curvature metrics on toric surfaces}, Geom. Funct. Anal. {\bf 19} (2009), 83--136.



 \bibitem{FutakiOnoWang}
  {\scshape A. Futaki, H. Ono, G. Wang}
	\emph{Transverse K\"ahler geometry of Sasaki manifolds and toric Sasaki--Einstein manifolds}, arXiv:math.DG/0607586.
%


 \bibitem{guan}
  {\scshape D. Guan}
    \emph{On modified Mabuchi functional and Mabuchi moduli space of K\"ahler metrics on toric bundles}, Math. Res. Lett. {\bf 6} (1999), 547--555.

 \bibitem{guillMET}
  {\scshape V. Guillemin}
    \emph{K\"ahler structures on toric varieties}, J. Diff. Geom. {\bf 40} (1994), 285--309.


 \bibitem{lejmi}
  {\scshape M. Lejmi}
    \emph{Extremal almost-Kahler metrics}, arXiv:math/0908.0859.

 \bibitem{TGQ}
  {\scshape E. Legendre}
    \emph{Toric geometry of convex quadrilaterals}, J. Symplectic Geom. (to appear), arXiv:math/0909.4512.

 \bibitem{LT:orbiToric}
  {\scshape E. Lerman, S. Tolman}
    \emph{Hamiltonian torus actions on symplectic orbifolds and toric varieties}, Trans. Amer. Math. Soc. {\bf 349} (1997), 4201--4230.

 \bibitem{L:contactCONVEX}
  {\scshape E. Lerman}
    \emph{A convexity theorem for torus actions on contact manifolds}, Illinois J. Math. {\bf 46} (2002), 171--184.

 \bibitem{L:contactToric}
  {\scshape E. Lerman}
    \emph{toric contact manifolds}, J. Symplectic Geom. {\bf 1} (2003), 785--828.

 \bibitem{reebMETRIC}
  {\scshape D. Martelli,  J. Sparks, S.-T. Yau}
    \emph{The geometric dual of {\it a}-- maximisation for toric Sasaki-Einstein manifolds}, Comm. Math. Phys.  {\bf 268}  (2006), 39--65.

 \bibitem{MSYvolume}
  {\scshape D. Martelli,  J. Sparks, S.-T. Yau}
    \emph{Sasaki-Einstein manifolds and volume minimisation}, Comm. Math. Phys. {\bf 280}  (2008), 611--673.

 \bibitem{tian:conj}
  {\scshape G. Tian}
    \emph{On Calabi's conjecture for complex surfaces with positive first Chern class}, Invent. Math. {\bf 101} (1990), 101--172.

 \bibitem{yau}
  {\scshape S.-T. Yau}
    \emph{Open problems in geometry}. Differential geometry: Partial differential equations on manifolds (Los Angeles, CA, 1990), 1--28, Proc. Sympos. Pure Math., 54, Part 1, Amer. Math. Soc., Providence, RI, 1993.


\end{thebibliography}

\end{document}